\documentclass[a4paper,11pt]{article}
\usepackage{geometry}
\usepackage[OT2,T1]{fontenc}
\usepackage{lipsum}
\usepackage[english]{babel}
\usepackage{amssymb,amsmath,amsthm,amscd,amsfonts}
\usepackage{mathrsfs}
\usepackage{tikz-qtree}
\usetikzlibrary{shapes,arrows.meta}
\usepackage{subfig}
\usepackage{tabularx}
\usepackage{calc} 
\usepackage{graphicx}
\usepackage[nottoc]{tocbibind}
\usepackage{makeidx}
\usepackage{authblk}
\makeindex
\usepackage{longtable}
\usepackage{enumerate}
\usepackage{array}
\usepackage{version}
\usepackage{lscape}
\usepackage{epigraph}
\usepackage[all]{xy}
\usepackage{multirow}
\usepackage{pgf,tikz}
\usepackage{hyperref}   
\usepackage{etoolbox}
\usepackage{lmodern}
\usepackage{cases}
\usepackage{xcolor}
\usepackage[font=scriptsize]{caption}
\apptocmd{\thebibliography}{\raggedright}{}{} 

\newtheoremstyle{one}
{11pt}
{11pt}
{\it}
{}
{\bf}
{.}
{1mm}
{}

\newtheoremstyle{two}
{11pt}
{11pt}
{}
{}
{\bf}
{.}
{1mm}
{}

\theoremstyle{one}
\newtheorem{theorem}{Theorem}[section]
\newtheorem{lemma}[theorem]{Lemma}

\newtheorem{proposition}[theorem]{Proposition}

\theoremstyle{two}
\newtheorem{question}[theorem]{Question}

\newtheorem{definition}[theorem]{Definition}
\newtheorem{example}[theorem]{Example}

\newtheorem{notation}[theorem]{Notation}
\newtheorem{remark}[theorem]{Remark}
\newtheorem*{proofofthmtorsion}{\textbf{Proof of Theorem \ref{thmtorsion}}}
\newtheorem*{proofofthm8lines}{\textbf{Proof of Theorem \ref{thm: case8lines}}}
\newtheorem*{questions}{Questions}

\def\Q{\mathbb{Q}}

\def\P{\mathbb{P}}
\def\E{\mathscr{E}}

\title{Torsion points and concurrent exceptional curves on del Pezzo surfaces of degree one}

\author{Julie Desjardins and Rosa Winter}\date{}

\begin{document}
\maketitle

\begin{center}\textbf{Keywords:}  del Pezzo surfaces of degree 1, torsion points, elliptic curve, \\ rational elliptic surface, exceptional curves\end{center}

\abstract{The blow-up of the anticanonical base point on a del Pezzo surface $S$ of degree 1 gives rise to a rational elliptic surface $\E$ with only irreducible fibers. The sections of minimal height of $\E$ are in correspondence with the $240$ exceptional curves on $S$. A natural question arises when studying the configuration of these curves: if a point on $S$ is contained in `many' exceptional curves, is it torsion on its fiber on $\E$? In 2005, Kuwata proved for the analogous question on del Pezzo surfaces of degree $2$, where there are 56 exceptional curves, that if `many' equals $4$ or more, the answer is yes. In this paper, we prove that for del Pezzo surfaces of degree 1, the answer is yes if `many' equals $9$ or more. Moreover, we give counterexamples where a \textsl{non}-torsion point lies in the intersection of $7$ exceptional curves. We give partial results for the still open case of 8 intersecting exceptional curves. 
}
\setcounter{tocdepth}{2}

\section{Introduction}

An elliptic surface $\mathscr{E}$ over a field $k$ is an projective variety of dimension 2 endowed with a fibration in elliptic curves $\pi:\mathscr{E}\longrightarrow \mathbb{P}^1$. In this paper, we also assume this fibration has at least one section, which guarantees that $\mathscr{E}$ admits a description as solution of a Weierstrass equation with polynomial coefficients, where the given section is set to be the zero-section. This object can thus alternatively be seen as an elliptic curve $\mathscr{E}_T$ over $k(T)$  whose rank we call the \textit{generic rank of $\mathscr{E}$}. A section of $\mathscr{E}$ corresponds to a $k(T)$-point of $\mathscr{E}$  and vice versa. We call each fiber $\mathscr{E}_t:=\pi^{-1}(t)$ of the fibration the \textit{specialization} at $t\in\P^1$. 

Silverman's specialization theorem predicts that when $k$ is a number field, then for all but finitely many $t\in\mathbb{P}^1$, the rank of the corresponding specialisation is at most the generic rank of $\mathscr{E}$: $$r_{k(T)}(\mathscr{E}_T)\geq r_k(\mathscr{E}_t).$$

Two natural questions are thus:

\begin{questions}
\begin{enumerate}
\item[]
\item When do we have a \textit{rank jump}, i.e. a value $t\in\mathbb{P}^1$ such that $r_{k(T)}(\mathscr{E}_T)<r_{k}(\mathscr{E}_t)$?
\item When do we have a \textit{rank fall}, i.e. a value $t\in\mathbb{P}^1$ such that $r_{k(T)}(\mathscr{E}_T)>r_{k}(\mathscr{E}_t)$?
\end{enumerate}
\end{questions}

While the rank jump question has been investigated by several different authors, much less is known about the second question. In this paper, although we do not study the rank fall question in its current form, we make progress toward the following variation: 

\begin{question}\label{q:manysectionsmeet}
When do `many' sections of an elliptic surface meet?
\end{question}

Sections of $\mathscr{E}$ that are non-torsion points in $\mathscr{E}_T(k)$ can intersect fibers of $\pi$ in non-torsion points, thus potentially contributing to the rank of the fiber. If multiple sections of infinite order of $\mathscr{E}$ intersect in one point, this potential contribution to the rank of the fiber on which the intersection occurs goes down (note that this does not automatically mean that there is a rank fall, since there might be points of infinite order on the same fiber, not coming from the specialization of a section). If the point at the intersection of many sections is torsion, the potential contribution of these sections to the rank drops to zero. Thus our final question is the following:

\begin{question}\label{q:meetattorsion}
When do `many' sections of $\E$ meet at a point which is torsion on its fiber?
\end{question}

Questions \ref{q:manysectionsmeet} and \ref{q:meetattorsion} are very difficult to solve on a general elliptic surface, and for this reason, we are looking at the case where $\E$ has a minimal model that is a del Pezzo surface of degree~1. In this case, $\mathscr{E}$ is rational. As mentioned in Remark \ref{r:density}, another motivation for us to study the questions for these surfaces is the distribution of the rational points on del Pezzo surfaces of degree~1.

\vspace{11pt}

A del Pezzo surface over a field $k$ is a smooth, projective, geometrically integral surface over $k$ with ample anticanonical divisor. The degree of a del Pezzo surface is the self-intersection number of the canonical divisor, which is an integer between 1 and 9. Over an algebraically closed field, a del Pezzo surface of degree $d$ is isomorphic to either $\mathbb{P}^1 \times\mathbb{P}^1$ (for $d = 8$), or to $\mathbb{P}^2$ blown up at $9-d$ points in general position \cite[Theorem 24.4.]{Manin}. Over non-algebraically closed fields, this is in general not true, and the arithmetic of these surfaces has been widely studied as we report in Remark \ref{rem: arithmetic DP's}. As is evident from the known results on the distribution of the rational points, the arithmetic complexity of del Pezzo surfaces goes up as the degree goes down. 

Over an algebraically closed field, a del Pezzo surface contains a finite number of exceptional curves, depending on the degree of the surface; we often call these \textsl{lines}.

For a del Pezzo surface $S$ of degree 1 and the corresponding elliptic surface $\E$ obtained by blowing up the base point of the anticanonical linear system on $S$, we ask the following. 

\begin{question}\label{Question}
If a point on $S$ is contained in `many' lines, is the corresponding point on $\E$ then torsion on its fiber?
\end{question}

Of course, `many' needs to be specified. In this paper we give a positive answer to this question for `many' equal to 9:

\begin{theorem}\label{thmtorsion}
If at least 9 exceptional curves on $S$ are concurrent in a point, then the corresponding point on $\E$ is torsion on its fiber. 
\end{theorem}

We also show that if we take `many' to be 7, the answer to this question is negative, at least in most characteristics, by providing two counterexamples (Examples~\ref{ex7lines} and \ref{ex7lines2}). These are the first known examples of a non-torsion point contained in more than 6 lines.

\vspace{11pt}

We call a set of exceptional curves on a del Pezzo surface \textsl{concurrent} in a point if that point is contained in all of them. 

\begin{remark}\label{r:density}For del Pezzo surfaces of degree 2, the situation is simpler, and a result similar to Theorem \ref{thmtorsion} is known. A del Pezzo surface of degree 2 is a double cover of $\mathbb{P}^2$ ramified along a quartic curve. On such a surface, a point is contained in at most 4 exceptional curves, and this happens exactly when its projection to $\mathbb{P}^2$ is in the intersection of 4 bitangents of the quartic curve. In \cite{K05}, Kuwata constructs elliptic surfaces from del Pezzo surfaces of degree 2. He shows that for a point contained in 4 exceptional curves on a del Pezzo surface of degree 2, the corresponding point on the elliptic surface is torsion on its fiber \cite[Proposition 7.1]{K05}. 

The situation for del Pezzo surfaces of degree 1 is more complex. First of all, outside characteristics 2 and 3 the maximal number of concurrent lines on a del Pezzo surface of degree~1 is 10 \cite[Theorems 1.1 and 1.2]{vLW}, but as Theorem \ref{thmtorsion} shows, a point contained in 9 lines is already torsion on its fiber. Moreover, 4 intersecting lines on a del Pezzo surface of degree 2 intersect pairwise with multipliciy 1, but there are a priori many different ways in which 9 or more concurrent lines on a del Pezzo surface of degree 1 can intersect; we explain this in Section \ref{sec: proof main thm}.
\end{remark}

The question of whether one can find an example with 8 lines on $S$ that intersect in a point which is non-torsion on its fiber of $\E$ stays unsolved. We show that the lines in such an example, if it exists, intersect each other according to one of 15 prescribed configurations in general, and 13 in characteristic 0.

\begin{theorem}\label{thm: case8lines}If 8 exceptional curves on a del Pezzo surface $S$ of degree 1 are concurrent in a point, then the corresponding point on the elliptic surface $\E$ obtained by blowing up the base point of the anticanonical linear system is torsion on its fiber, except possibly in the following case.  Consider the graph $\mathcal{G}$ where each vertex corresponds to one of the 8 exceptional curves and two vertices are connected with an edge if and only if the corresponding exceptional curves intersect with multiplicity 2. Then the 8 exceptional curves intersect pairwise with multiplicities 1 or 2, and the graph $\mathcal{G}$ equals one of the graphs with numbers $1,\ldots8,10,\ldots,15,19$ in Figure~\ref{table cliques}. In characteristic 0, numbers 7 and 15 can be excluded from this list.
\end{theorem}

\begin{remark}\label{rem: arithmetic DP's}
Our original motivation comes from the distribution of rational points on del Pezzo surfaces.

Del Pezzo surfaces of degree at least 2 over a field $k$ with a $k$-rational point are known to be $k$-unirational under the extra condition for degree 2 that the $k$-rational point lies outside a closed subset \cite{Se43,Se51,Manin,Ko02,Pie,STVA}. Del Pezzo surfaces of degree 1 over a field with characteristic not 2 are known to be unirational if they admit a conic bundle structure \cite{KM17}, but outside this case, there is no example
of a minimal del Pezzo surface of degree 1 that is known to be
$k$-unirational, nor of one that is known not to be $k$-unirational. Unirationality for del Pezzo surfaces of degree 1 in general is considered an extremely difficult problem. Over infinite fields there are several partial results on Zariski density of the set of rational points on these surfaces \cite{VA11,Ulas2,Ulas,Jabara,SVL,BulthuisVanLuijk,DW}, which is a weaker notion in the sense that it is implied by unirationality. However, there are still many families of del Pezzo surfaces of degree 1 for which even Zariski density of rational points is not proven.

For several of the earlier mentioned results, one requires the existence of a rational point on the surface which is not contained in too many lines. For example, in the paper \cite{STVA} the authors show that a del Pezzo surface of degree 2 is unirational if it contains a point that is not contained in 4 lines, and lies outside a specific subset of the surface. As another example, one of the conditions for the set of rational points on a del Pezzo surface of degree 1 to be dense in \cite{SVL} is the existence of a point not contained in 6 of these lines. At the same time, several of the results on density of rational points on a del Pezzo surface of degree 1 require the existence of a point which is non-torsion on its fiber of the elliptic surface obtained by blowing up the base point of the anticanonical linear system; see for instance \cite{SVL,DW}. The paper \cite{SVL} contains several examples of a point on a del Pezzo surface of degree 1 for which their method fails, and in all cases, the point is contained in the intersection of at least 6 lines, and it is torsion on its fiber. This observation motivates Question \ref{Question}.
\end{remark}

The paper is organized as follows. In Section \ref{Sec: main results} we present the necessary background on the exceptional curves on a del Pezzo surface $S$ of degree 1, the relation to the root system \textbf{E}$_8$, the elliptic surface $\E$ obtained from $S$, and the strict transforms on $\E$ of the exceptional curves on $S$. In Section \ref{sec: case of 7} we show that 7 concurrent lines do not always intersect in a torsion point, by giving two counterexamples (Examples \ref{ex7lines} and \ref{ex7lines2}). In Section \ref{sec: proof main thm} we prove Theorem \ref{thmtorsion}. In Section \ref{Sec: 8lines} we study the case of $8$ lines and prove
Theorem~\ref{thm: case8lines}. 

\vspace{11pt}

Part of this paper appeared in the PhD thesis of the second author. Specifically, Sections \ref{Sec: main results} and~\ref{sec: proof main thm} are slight modifications of \cite[Chapter 5 and Sections 1.4.2, 1.4.3]{Wthesis}. The innovative material compared to what is found in the thesis is Sections \ref{sec: case of 7} and \ref{Sec: 8lines}, which includes Theorem \ref{thm: case8lines}.

\vspace{11pt}

Computations were done in \texttt{magma} \cite{MR1484478}; the code is available online \cite{MagmaCode}.

\subsection{Acknowledgements}

We started thinking about this question a while ago, and we are grateful to Marc Hindry, Cec\'{i}lia Salgado and Anthony V\'arilly-Alvarado for helpful discussion in formulating our problem. The ideas of Ronald van Luijk and Jennifer Park contributed greatly to the proof of Theorem \ref{thmtorsion} when we discussed this question with them during the trimester Reinventing Rational Points at IHP in Paris. The first author is partially supported by an NSERC Discovery Grant. Many of the computations were done on the compute servers of the Max Planck Institute for Mathematics in the Sciences in Leipzig, while the second author was being employed there. While writing the paper, the second author was supported by UKRI fellowship MR/T041609/2.

\section{Background}\label{Sec: main results}

In this Section we give the necessary background for the rest of this paper.

\subsection{Exceptional curves and the $E_8$ root system}

\begin{definition}
Let $r\leq8$ be an integer, and let $P_1,\ldots,P_r$ be points in $\mathbb{P}^2$. 
We say that $P_1,\ldots,P_r$ are in \textsl{general position} if there is no line containing three of the points, no conic containing six of the points, and no cubic containing eight of the points with a singularity at one of them. 
\end{definition}

As mentioned in the introduction, a del Pezzo surface over an algebraically closed field is isomorphic to either $\mathbb{P}^1 \times\mathbb{P}^1$ (for $d = 8$), or to $\mathbb{P}^2$ blown up at $9-d$ points in general position \cite[Theorem 24.4]{Manin}. 

\begin{definition}An \textsl{exceptional curve} on a del Pezzo surface $S$ with canonical divisor $K_S$ is an irreducible projective curve $C\subset S$ such that $$C^2=C\cdot K_S=-1.$$
\end{definition}

We often call exceptional curves \textsl{lines}, since for a del Pezzo surface of degree $d\geq3$, the images of these curves under the anticanonical embedding in $\mathbb{P}^d$ are lines. Over an algebraically closed field, we know exactly how to describe the lines on a del Pezzo surface in terms of curves in~$\P^2$.

\begin{theorem}[{\cite[Theorem 26.2]{Manin}}]\label{thm: exc curves P2} For an integer $d\in\{1,\ldots,8\}$, let $P_1,\dots,P_{9-d}$ be $9-d$ points in general position in~$\P^2$. The exceptional curves on the del Pezzo surface of degree $9-d$ obtained by the blow-up of $P_1,\ldots,P_{9-d}$ are 
\begin{itemize}
\item the exceptional curves $E_i$ above the points $P_i$ for $i\in\{1,\ldots,9-d\}$,
\end{itemize}
and the strict transforms of the following curves in $\mathbb{P}^2$.
\begin{itemize}
	\item The line $L_{i,j}$ passing through the points $P_i$ and $P_j$, for $i\not=j$,
	\item the conics passing through five of the points,
	\item the cubic $C_{i,j}$ not passing through $P_j$, passing twice through $P_i$ and passing once through the six remaining points,
	\item the quartic $Q_{i,j,k}$ passing through the eight points with a double point in $P_i$, $P_j$ and $P_k$ for $i,j,k$ distinct,
	\item the quintics passing through the eight points with double points at 6 of them, and
	\item the sextics passing through the eight points with double points at 7 of them, and a triple point at one of them.
\end{itemize}
\end{theorem}

\begin{notation}\label{notation curves P^2}
Throughout the paper, for $r$ points $P_1,\ldots,P_r$ in general position in $\mathbb{P}^2$, and for $i\in\{1,\ldots,r\}$, we use the notation $E_i$ for the exceptional curve above $P_i$ on the del Pezzo surface obtained by blowing up $P_1,\ldots,P_r$. Similarly, for $i,j,k\in\{1,\ldots,r\}$ we write $L_{i,j}$, $C_{i,j}$, $Q_{i,j,k}$ for the lines, cubics, and quartics in $\mathbb{P}^2$, as defined in Theorem~\ref{thm: exc curves P2}.
\end{notation}

From now on, we focus on del Pezzo surfaces of degree 1. Let $S$ be such a surface over an algebraically closed field. From Theorem~\ref{thm: exc curves P2} it follows that $S$ contains 240 exceptional curves. These are in one-to-one correspondence with the root system \textbf{E}$_8$, as we will now describe.

\vspace{5pt}

Let $\langle\cdot,\cdot\rangle$ be the negative of the intersection pairing on Pic $S$, and let $K_S$ be the canonical divisor of $S$. Then $\langle\cdot,\cdot\rangle$ on $\mathbb{R}\otimes_{\mathbb{Z}}\mbox{Pic }S$ induces the structure of a Euclidean space on the orthogonal complement $K_S^{\perp}$ of the class of the canonical divisor, and with this structure, the set $$R=\{D\in\mbox{Pic }S\;|\;\langle D,D\rangle=2;\;D\cdot K_{S}=0\}$$ is a root system of type \textbf{E}$_8$ in $K_{S}^{\perp}$ \cite[Theorem 23.9]{Manin}. 
Let $I$ be the set of the 240 exceptional curves in Pic $S$. For $e\in I$ we have $e+K_{S}\in K_{S}^{\perp}$ and $\langle e+K_{S},e+K_{S}\rangle=2$, and this gives a bijection \begin{equation}\label{eq:bijection C E}
I\longrightarrow R,\;\;\; e\longmapsto e+K_{S}.\end{equation} For $e_1,e_2\in I$ we have $\langle e_1+K_{S},e_2+K_{S}\rangle=1-e_1\cdot e_2$. As a consequence of this bijection, the group of permutations of $I$ that preserve the intersection pairing is isomorphic to the Weyl group $W_8$, which is the group of permutations of \textbf{E}$_8$ generated by the reflections in the hyperplanes orthogonal to the roots \cite[Theorem 23.9]{Manin}. Another way of stating the bijection (\ref{eq:bijection C E}) is to note that the \textsl{weighted graphs} on $I$ and \textbf{E$_8$} and their automorphism groups are isomorphic (Remark \ref{remgammaG}).

\begin{definition}\label{defgraph}
By a \textsl{graph} we mean a pair $(V,D)$, where $V$ is a set of elements called \textsl{vertices}, and $D$ a subset of the power set of $V$ such that every element in $D$ has cardinality 2; elements in $D$ are called \textsl{edges}, and the \textsl{size} of the graph is the cardinality of~$V$. A graph $(V,D)$ is \textsl{complete} if for every two distinct vertices $v_1,v_2\in V$, the pair $\{v_1,v_2\}$ is in $D$.\\
By a \textsl{weighted graph} we mean a graph $(V,D)$ with a map $\psi\colon D\longrightarrow A$, where $A$ is any set, whose elements we call \textsl{weights}; for any element $d$ in $D$ we call $\psi(d)$ its weight. If $(V,D)$ is a weighted graph with weight function~$\psi$, then we define a \textsl{weighted subgraph} of $(V,D)$ to be a graph $(V',D')$ with map $\psi'$, where $V'$ is a subset of $V$, while $D'$ is a subset of the intersection of $D$ with the power set of $V'$, and $\psi'$ is the restriction of $\psi$ to $D'$. A \textsl{clique} of a weighted graph is a complete weighted subgraph. Its \textsl{size} is the number of vertices.\\
An \textsl{isomorphism} between weighted graphs $(V,D)$ and $(V',D')$ with weight functions $\psi\colon D\longrightarrow A$ and $\psi'\colon D'\longrightarrow A'$, respectively, consists of a bijection $f$ between the sets $V$ and $V'$ and a bijection $g$ between the sets $A$ and $A'$, such that for any two vertices $v_1,v_2\in V$, we have $\{v_1,v_2\}\in D$ with weight $w$ if and only if $\{f(v_1),f(v_2)\}\in D'$ with weight $g(w)$. We call the map $f$ an \textsl{automorphism} of $(V,D)$ if $(V,D)=(V',D')$, $\psi=\psi'$, and $g$ is the identity on $A$.    
\end{definition}

\begin{definition}\label{defgammaG}
By $\Gamma$ we denote the complete weighted graph whose vertex set is the set of roots in~\textbf{E$_8$}, and where the weight function is induced by the dot product.  
Similarly, by $G$ we denote the complete weighted graph whose vertex set is $I$, and where the weight function is the intersection pairing in Pic $S$. 
\end{definition}

\begin{remark}\label{remgammaG}
There is an isomorphism of weighted graphs between $G$ and $\Gamma$, that sends a vertex $e$ in $G$ to the corresponding vertex $e+K_S$ in $\Gamma$, and an edge $d=\{e_1,e_2\}$ in~$G$ with weight~$w$ to the edge $\delta=\{e_1+K_S,e_2+K_S\}$ in $\Gamma$ with weight $1-w$. The different weights that occur in $G$ are $0,1,2,$ and $3$, and they correspond to weights $1,0,-1,$ and~$-2$, respectively, in $\Gamma$. As a consequence, the weighted graphs $G$ and $\Gamma$ have isomorphic automorphism groups, given by the Weyl group $W_8$.
\end{remark}

\subsection{The elliptic surface}

Let $S$ be a del Pezzo surface of degree 1 over a field $k$ with canonical divisor $K_S$. The surface $S$ can be embedded in the weighted projective space $\mathbb{P}(2,3,1,1)$ with coordinates $(x:y:z:w)$ as the set of solutions to the equation \begin{equation}\label{eqrem}y^2+a_1xy+a_3y-x^3-a_2x^2-a_4x-a_6=0,
\end{equation}
where $a_i\in k[z,w]$ is homogeneous of degree $i$ for each $i$ in $\{1,\ldots,6\}$. The linear system $|{-K}_S|$ induces a rational map $S\dasharrow\mathbb{P}^1,\;$ $(x:y:z:w)\mapsto(z:w)$, which is defined everywhere except in the base point of $|{-K}_S|$, given by $\mathcal{O}=(1:1:0:0)$. The blow-up of $S$ in $\mathcal{O}$ gives a rational elliptic surface $\E$ with only irreducible fibres,\footnote{Reciprocally, the contraction of the zero-section on a rational elliptic surface $\E$ produces a del Pezzo surface of degree~1 whenever $\E$ has only irreducible fiber \cite[Corollaire 1.2.9, Lemme 1.2.10.]{Desjardinsthese}.} and we denote this blow-up by $\pi\colon\E\longrightarrow S$. We denote the induced elliptic fibration on $\E$ by $\nu\colon\E\longrightarrow\mathbb{P}^1$. The generic fiber of $\E$ is an elliptic curve $E$ over the function field $k(t)$ of $\mathbb{P}^1$. We call the Mordell--Weil group of $E$ the \textsl{Mordell--Weil group of $\E$}. Points in $E(k(t))$  correspond to sections of $\nu$ that are defined over $k$ \cite[Proposition 3.10]{Sil94}. For $(z_0:w_0)\in\mathbb{P}_k^1$, the fiber $\nu^{-1}((z_0:w_0))$ is isomorphic to the cubic curve in $\mathbb{P}_k^2$ with affine Weierstrass equation 
\begin{equation}\label{wsfiber} Y^2+a_1(z_0,w_0)XY+a_3(z_0,w_0)Y=X^3+a_2(z_0,w_0)X^2+a_4(z_0,w_0)X+a_6(z_0,w_0).
\end{equation}

The point at infinity on such a fiber is the intersection on $\E$ with the exceptional divisor $\tilde{\mathcal{O}}$ above $\mathcal{O}$. For a point $P\in S\setminus\{\mathcal{O}\}$ we denote by $P_{\E}$ the corresponding point on $\E$.

\begin{remark}\label{remramcurve}The linear system $|-2K_{S}|$ of the bi-anticanonical divisor of $S$ induces a morphism~$\varphi$, which is the composition of the projection of $S$ to $\mathbb{P}(2, 1, 1)$ on the $x,z,w$-coordinates, and the 2-uple embedding of $\mathbb{P}(2, 1, 1)$ in $\mathbb{P}^3$. This morphism realizes $S$ as a double cover of a cone in $\mathbb{P}^3$ ramified over a sextic curve. Using the notation in (\ref{eqrem}), the morphism $\varphi$ is ramified at the points $(x_0:y_0:z_0:w_0)\in S$ for which we have $2y_0+a_1x_0+a_3=0$, and from (\ref{wsfiber}) it follows that these are exactly the points that are 2-torsion on their fiber on $\E$.
\end{remark}

\begin{remark}\label{exceptional curves are sections}
Since exceptional curves on $S$ are defined over a separable closure of $k$ \cite[Theorem 2.1.1]{VAthesis}, from \cite[Theorem 1.2]{VA08} it follows that the exceptional curves on $S\subset\mathbb{P}(2,3,1,1)$ are exactly the curves given by $$x=p(z,w),\;\;\;y=q(z,w),$$ where $p,q\in k[z,w]$ are homogeneous of degrees 2 and 3, respectively. Note that this implies that an exceptional curve never contains $\mathcal{O}=(1:1:0:0)$. Therefore, for an exceptional curve $C$ on~$S$, its strict transform $\pi^*(C)$ on~$\E$ satisfies $$\pi^*(C)^2=-1,\;\;\;\;\;\pi^*(C)\cdot K_{\E}=\pi^*(C)\cdot(\pi^*(K_{S})+\tilde{\mathcal{O}})=-1+0=-1.$$ Thus $\pi^*(C)$ is an exceptional curve on $\E$ as well, and, since a fiber of $\nu$ is linearly equivalent to $-K_{\E}$,
the curve $\pi^*(C)$ intersects every fiber once. This gives a section of~$\nu$.
\end{remark}

\begin{remark}\label{rem: rank 8 so 9 gives torsion}
Theorem \ref{thmtorsion} seems intuitively true by the following argument, which was pointed out to us by several people. Let $P$ be a point on $S$ that is contained in at least 9 exceptional curves, say $L_1,\ldots,L_n$. These curves correspond to sections $\tilde{L}_1,\ldots,\tilde{L}_n$ of $\E$ (Remark \ref{exceptional curves are sections}), which in turn correspond to elements in the Mordell--Weil group of $\E$. This Mordell--Weil group has rank at most 8 over $k$ \cite[Theorem 10.4]{Shioda_MW}, so in this group there must be a relation $a_1\tilde{L}_1+\cdots+a_n\tilde{L}_n=0$, where $a_1,\ldots,a_n\in \mathbb{Z}$ are not all zero. Since all $n$ exceptional curves contain the point $P$, this specializes to $(a_1+\cdots+a_n)P_{\E}=0$ on the fiber of $P$ on $\E$. If one reasons too quickly, it seems that this proves that $P_{\E}$ is torsion of order dividing $a_1+\cdots +a_n$ on its fiber. However, it might be the case that $a_1+\cdots +a_n=0$, so this does not prove Theorem \ref{thmtorsion}. The key part in our proof is therefore that we show, using results from \cite{WvL}, that there is always a relation between $\tilde{L}_1,\ldots,\tilde{L}_n$ in the Mordell--Weil group of $\E$ that specializes to a \textsl{non-trivial} relation on the fiber of~$P_{\E}$; see Lemma \ref{rankandkernel}. 
\end{remark}

\section{The case of 7 lines: two examples}\label{sec: case of 7}
In this section we give two examples that show that for a point $P$ in the intersection of 7 exceptional curves on a del Pezzo surface of degree 1, the point $P_{\E}$ is not guaranteed to be torsion on its fiber. This gives a negative answer to Question \ref{Question} for `many' equal to 7 or less. As explained in Remark~\ref{rem: 10 remaining char's}, our examples hold in all but 10 characteristics.

\begin{example}\label{ex7lines}
Let $S$ be the blow-up of $\mathbb{P}_{\mathbb{Q}}^2$ in the eight points:
\begin{align*}
&P_1=( 0: 1: 1 ); & &P_2=(0:14:13);\\
&P_3=(1:0:1); & &P_4=(21: 0: 13);\\
&P_5=(1: 1: 1); & &P_6=(6: 6: -1 );\\
&P_7=(-2: 2:1); & &P_8=(-3: 3: -1).
\end{align*}
It is easy to check that these points are in general position, thus $S$ is a del Pezzo surface of degree~1. Consider the following curves in $\mathbb{P}^2$ (see Notation \ref{notation curves P^2}): the line $L_{1,2}$ given by $x=0$, the line $L_{3,4}$ given by $y=0$, the line $L_{5,6}$ given by $x-y=0$, the line $L_{7,8}$ given by $x+y=0$, the cubic $C_{1,2}$ given by 
$$26x^3+42x^2y-68x^2z-33xy^2-9xyz+42xz^2-36y^3+72y^2z-36yz^2=0, $$
the cubic $C_{3,4}$ given by 
 $$36x^3 + 46x^2y - 72x^2z - 42xy^2 - 4xyz + 36xz^2 - 39y^3 + 81y^2z - 42yz^2=0,$$
and the quartic $Q_{2,6,7}$ given by\begin{multline*}
1144x^4 + 1288x^3y - 808x^3z - 2910x^2y^2 + 4748x^2yz - 3864x^2z^2 -
    1092xy^3 + 318xy^2z - 2352xyz^2\\
     + 3528xz^3 + 1521y^4 - 4797y^3z
    + 5040y^2z^2 - 1764yz^3=0.\end{multline*}
            These 7 curves all go through $Q=(0:0:1)$, and each of them gives rise to an exceptional curve on $S$ by Theorem \ref{thm: exc curves P2}. The base point $\mathcal{O}$ of the anticanonical linear system of $S$ is strict transform of the base point in $\P^2$ of the pencil of cubics through $P_1,\ldots,P_8$, which is given by 
$$B=(27 : 68: 109) .$$
The fiber of $Q$ of the elliptic surface obtained by blowing up $S$ in $\mathcal{O}$ is given by the strict transform of the cubic curve through $P_1,\dots,P_8,Q$ with origin given by $B$. With \texttt{magma} we check that the point $Q$ is non-torsion on this elliptic curve. 
\end{example}

\begin{remark}\label{rem: 10 remaining char's}
The previous example also holds over any other field $k$, as long as the characteristic of $k$ is $p$ for all but a finite number of primes $p$. In fact, the only characteristics for which this does not hold are the ones for which $P_1,\ldots,P_8$ are not in general position, and for which the fiber of $Q$ is not an elliptic curve, i.e., for which it is singular. We compute this with \texttt{magma} and find the primes $\{ 2, 3, 5, 7, 11, 13, 17, 19, 23, 29, 31, 41, 3319 \}$. 
It is not hard to generate similar examples that hold in some of the missing characteristics; for example, the eight points in $\mathbb{P}^2$ given by Example~\ref{ex7lines2} are in general position with a non-singular fiber in all but 29 characteristics, and this gives, together with Example \ref{ex7lines}, examples of 7 exceptional curves that are concurrent in a point $P$ such that $P_{\E}$ is not torsion on its fiber for each characteristic except for $p=2,3,5,7,11,13,17,19,23,29$.
\end{remark}

\begin{example}\label{ex7lines2}
Let $S$ be the blow-up of $\mathbb{P}_{\mathbb{Q}}^2$ in the eight points:
\begin{align*}
&P_1=( 0: 1: 1 ); & &P_2=(0:3861:1957);\\
&P_3=(1:0:1); & &P_4=(1188: 0: -19);\\
&P_5=(1: 1: 1); & &P_6=(780: 780: 1883 );\\
&P_7=(-52: 52: 51); & &P_8=(-9: 9: -17).
\end{align*}
It is an easy check that $P_1,\ldots,P_8$ are in general position, so $S$ is a del Pezzo surface of degree 1. Again the line $L_{1,2}$ is given by $x=0$, the line $L_{3,4}$ by $y=0$, the line $L_{5,6}$ by $x-y=0$, and the line $L_{7,8}$ by $x+y=0$. The cubic $C_{1,2}$ is now given by 
\begin{multline*}247x^3 - 15444x^2y +15197x^2z - 56500xy^2 +71944xyz \\
- 15444xz^2 - 24336y^3 +48672y^2z - 24336yz^2=0,\end{multline*}the cubic $C_{3,4}$ is given by 
\begin{multline*}24336x^3+48425x^2y-48672x^2z+15444xy^2\\
-63869xyz+24336xz^2+7828y^3-23272y^2z+15444yz^2=0,\end{multline*}and the quartic $Q_{2,6,7}$ is given by  \begin{multline*}2705155115x^4-160214640456x^3y+165198460765x^3z-340717645684x^2y^2\\
+583405507724x^2yz-245421685080x^2z^2-86417174688xy^3+301295315984xy^2z\\
-297351362880xyz^2+77518069200xz^3-6127758400y^4+30306884800y^3z\\
-48030840000y^2z^2+23851713600yz^3=0.\end{multline*}These 7 curves again all go through $Q=(0:0:1)$, and completely analogously to the previous example we check with \texttt{magma} that the point $Q$ is non-torsion on its fiber. 
\end{example}

\begin{remark}
The first steps of the process used to find Examples \ref{ex7lines} and \ref{ex7lines2} allow us to find a point $P$ on $S$ contained in 6 lines which is non-torsion. For instance, let $k$ be a field of characteristic~0, and consider the following points in $\mathbb{P}^2_k$.
\begin{align*}
& P_1=(0: 1: 1); & &P_2=( 0: 319: -920 );\\
&P_3=(1: 0: 1); & &P_4=(799: 0: 610);\\
&P_5=( 1: 1: 1 ); & &P_6=( 1: 1: -1);\\
&P_7=( 31: 1: 123); & &P_8=(31: 1: 11).\end{align*}
The curves $L_{1,2}$, $L_{3,4}$, $L_{5,6}$, $L_{7,8}$, $C_{1,2}$, and $C_{3,4}$ in $\P^2$ all go through $Q=(0:0:1)$, and each of them gives rise to an exceptional curve on $S$.
This is not the first example of  a non-torsion point contained in 6 lines, see \cite[Example 5.1.5]{Wthesis}. However, the 6 lines in this example have a different intersection graph than those in \cite[Example 5.1.5]{Wthesis}, showing that there are several families of such examples.
Comparatively, there are `many more' examples with 6 lines than 7 lines. While our method leads without trouble to more examples of 6 concurrent lines, the further conditions required by the construction of the $7^{\mbox{\tiny{th}}}$ line lead to a much sparser set of possibilities to obtain more examples like \ref{ex7lines} and \ref{ex7lines2}. 

\end{remark}

\section{The case of 9 or more lines: proof of Theorem \ref{thmtorsion}}\label{sec: proof main thm}
In this Section we prove Theorem \ref{thmtorsion}. Let $S$ be a del Pezzo surface of degree 1 over a field $k$, let $\E$ be the corresponding elliptic surface, and $E$ the generic fiber of $\E$. We start by describing a pairing on the Mordell--Weil group of~$\E$.

\vspace{11pt}
 
Let $\varphi:S\rightarrow\mathbb{P}^3$ be the morphism induced by the bi-anticanonical linear system on $S$ as in Remark~\ref{remramcurve}. Let $e_1,\ldots,e_n$ be at least 9 exceptional curves on $S$ that are concurrent in a point~$Q$ that lies outside the ramification curve of $\varphi$. Let $L_1,\ldots,L_n$ be the corresponding sections of $\nu :\E\rightarrow\mathbb{P}^1$. Let $\langle\cdot,\cdot\rangle_h$ be the symmetric and bilinear pairing on the Mordell--Weil group of $\E$ as defined in \cite[Theorem 8.4]{Shioda_MW}; that is, for $C_1,C_2$ in $E(k(T))$, we have $\langle C_1,C_2\rangle_h =-(\varphi_h(C_1)\cdot\varphi_h(C_2))$, where $\varphi_h\colon E(k(T))\longrightarrow\mbox{ Pic }\E$ is the map given in \cite[Lemmas 8.1 and~8.2]{Shioda_MW}, and $\cdot$ is the intersection pairing in the Picard group of~$\E$. We call $\langle\cdot,\cdot\rangle_h$ the \textsl{height pairing} on $E(k(T))$. 

\begin{lemma}\label{heightpairingisdotproduct}
For two exceptional curves in Pic $S$, the height pairing of the corresponding sections in the Mordell--Weil group of $\E$ is the same as the dot product of the roots in the root system \textbf{E$_8$} associated to these exceptional curves under the bijection (\ref{eq:bijection C E}).
\end{lemma}
\begin{proof}This statement follows directly from the isomorphism between the Mordell--Weil lattice of $\E$ and \textbf{E$_8$}; see Remark \ref{rem: isomorphism of lattices}. We here illustrate how one can also see it from the properties of the height pairing.
Let $C_1,C_2$ be two sections of $\E$ that are strict transforms of exceptional curves $c_1,c_2$ in $S$. Since $\E$ has no reducible fibers, by \cite[Lemma 8.1]{Shioda_MW} we have $$\varphi_h(C_1)\cdot\varphi_h(C_2)=([C_1]-[\tilde{\mathcal{O}}]-F)\cdot([C_2]-[\tilde{\mathcal{O}}]-F),$$
where $[C_1],[C_2],[\tilde{\mathcal{O}}]$ are the classes of $C_1,C_2$, and the zero section, respectively, and $F$ is the class of a fiber. This gives  
$$\varphi_h(C_1)\cdot\varphi_h(C_2)=[C_1]\cdot[C_2]-1,$$ where we use that the zero section is an exceptional curve, and it is disjoint from $C_1$ and $C_2$ (Remark \ref{exceptional curves are sections}). We conclude that we have $\langle C_1,C_2\rangle_h =1-[C_1]\cdot[C_2]$. Since $C_1,C_2$ are disjoint from~$\tilde{\mathcal{O}}$, the intersection pairing of $C_1$ and $C_2$ in Pic $\E$ is the same as the intersection pairing of $c_1$ and $c_2$ in Pic $S$. The statement now follows from the bijection (\ref{eq:bijection C E}). \end{proof}

\begin{remark}\label{rem: isomorphism of lattices}
Lemma \ref{heightpairingisdotproduct} is true for every section of the elliptic fibration, not only the ones coming from the lines on the del Pezzo surfaces of degree 1. The key is to see that $E(\overline{\Q}(t))/E(\overline{\Q}(t))_{tors}$ and $NS(\E)/T$ are isomorphic as lattices, where $NS(\E)$ is the N\'eron-Severi lattice and $T$ is the trivial sublattice of $NS(\E)$ (generated by the zero section and fibre components). Moreover, the latter is isomorphic to $K^{\perp}_{S}$. See for instance in \cite[Section 7.4.]{VAZ} for a proof of these facts. We thank an anonymous referee for pointing this out to us.
\end{remark}

Let $M$ be the \textsl{height pairing matrix} of $e_1,\ldots,e_n$, that is, $M$ is the $n\times n$ matrix with entries $M_{ij}=\langle L_i,L_j\rangle_h$ for $i,j\in\{1,\ldots,n\}$.

\begin{lemma}\label{rankandkernel}
The kernel of the matrix $M$ contains a vector $(a_1,\ldots,a_n)$ in $\mathbb{Z}^n$ with $a_1+\cdots+a_n\neq0$.
\end{lemma}
\begin{proof}
Recall the complete weighted graphs $G$ and $\Gamma$ as defined in Definition \ref{defgammaG}. Since $Q$ lies outside the ramification curve of $\varphi$, where $\varphi$ is the morphism in Remark \ref{remramcurve}, the exceptional curves $e_1,\ldots,e_n$ correspond to a clique of size $n$ in $G$ that is contained in a maximal clique in $G$ with only edges of weights 1 and 2 \cite[Remark 2.11]{vLW}. The latter corresponds to a maximal clique~$C$ in $\Gamma$ with only edges of weights -1 and 0 by the bijection (\ref{eq:bijection C E}). Since  $n\geq9$, the clique~$C$ has size at least 9. The table in \cite[Appendix A]{WvL} contains all isomorphism types of maximal cliques in $\Gamma$ with only edges of weights -1 and 0 and of size at least 9 \cite[Proposition 21]{WvL}; there are 11 maximal cliques of size~9, which we call $\alpha_1,\ldots,\alpha_{11}$ in the order that they appear in the table, there are~6 maximal cliques of size 10, which we call $\beta_1,\ldots,\beta_6$ in the order that they appear in the table, and there is 1 maximal clique of size 12, which we call $\gamma$. For each of these 18 cliques, whose elements correspond to roots in \textbf{E}$_8$, we compute its Gram matrix, which is the matrix where the entry $(i,j)$ is the dot product of the roots corresponding to the $i$-th and $j$-th vertex in the clique after choosing an ordering on the vertices. With \texttt{magma} we find the generators for the kernels of these matrices. The results are in Table \ref{kernels}. Let $r$ be the number of vertices of $C$, and let $N$ be the Gram matrix of~$C$; then the kernel of $N$ is equal to one of the 18 kernels in the table, after rearranging the order of the vertices in $C$ if necessary. Since we have $n\geq9$, we see from Table~\ref{kernels} that for any subset of $n$ vertices in $C$, there is a vector $(a_1,\ldots,a_r)$ in the kernel of $N$ which is~0 outside the entries corresponding to the $n$ vertices, and such that $a_1+\cdots+a_r\neq0$. By Lemma~\ref{heightpairingisdotproduct}, this gives a vector in the kernel of $M$ as claimed.
\end{proof} 

\begin{table}[!h]
\begin{center}
\begin{tabular}{c|c}
Clique & Basis for the kernel\\
\hline
$\alpha_1$ & $\{(1, 1, 0, 0, 0, 0, 1, 0, 1),(0, 0, 1, 1, 1, 1, 0, 2, 0)\}$\\
\hline
$\alpha_2$ & $\{(1, 0, 1, 0, 0, 1, 0, 0, 1),(0, 0, 0, 1, 1, 0, 1, 0, 0)\}$\\
\hline
$\alpha_3$ & 
$\{(1, 1, 1, 0, 0, 1, 0, 0, 1),(0, 0, 0, 1, 1, 0, 1, 1, 0)\}$\\ 
\hline
$\alpha_4$ & 
$\{(1, 1, 0, 1, 0, 0, 1, 0, 1),(0, 0, 1, 0, 0, 1, 0, 1, 0)\}$\\ 
\hline
$\alpha_5$ & $\{(2, 1, 1,   0 ,  2 ,  0,   0, 1, 1),(0, 0, 0, 1, 0, 1, 1, 0, 0)
\}$\\ 
\hline
$\alpha_6$ & $\{    (1, 1, 1, 1, 1, 1, 1, 1, 1)\}$\\
\hline
$\alpha_7$ & 
$\{    (1, 1, 1, 0, 1, 1, 1, 1, 1)\}$\\ 
\hline
$\alpha_8$ & $   \{ (0, 1, 1, 2, 2, 2, 1, 1, 0)\}$\\
\hline
$\alpha_9$ & $\{    (2, 1, 1, 1, 1,   2 ,  2,   2,   2))\}$\\
\hline
$\alpha_{10}$ & 
$\{  ( 2,   2,   0, 3, 1,   4,   2, 3, 1)\}$\\
\hline
$\alpha_{11}$ & $\{ (  6, 3, 1, 4, 4, 2, 2, 5, 3)\}$\\ 
\hline
$\beta_1$ & $\{(1, 0, 1, 0, 0, 2, 1, 0, 0, 1),(0, 1, 0, 1, 2, 0,0, 1, 1, 0)\}$\\
\hline
$\beta_2$ & $\{(1, 1, 0, 0, 0, 0, 0, 0, 1, 1),(0, 0, 0, 0, 1, 1, 1, 1, 0, 0)\}$\\
\hline
$\beta_3$ & $\{(1, 1, 0, 1, 0, 0, 0, 1, 0, 1),(0, 0, 1, 0, 1, 1, 1, 0, 1, 0)\}$\\
\hline
$\beta_4$ & $\{(1, 1, 0, 1, 0, 1, 0, 0, 1, 1),(0, 0, 0, 0, 1, 0, 1, 1, 0, 0)\}$\\
\hline
$\beta_5$ & $\{(1, 1, 0, 0, 0, 0, 0, 0, 1, 1),(0, 0, 1, 1, 1, 1, 2, 2, 0, 0)\}$\\
\hline
$\beta_6$ & $\{(2, 1, 3, 0, 2, 0, 2, 0, 1, 1),(0, 0, 0, 1, 0, 1, 0, 1, 0, 0)\}$\\
\hline
$\gamma$ & $\{(1, 1, 0, 0, 0, 0, 0, 0, 0, 0, 0, 1),(0, 0, 1, 0, 0, 1, 0, 0, 0, 0, 1, 0),$\\
&$(0, 0, 0, 1, 0, 0, 0, 1, 1, 0, 0, 0),(0, 0, 0, 0, 1, 0, 1, 0, 0, 1, 0, 0)\}$
\end{tabular}
\caption{Bases}
\label{kernels}
\end{center}
\end{table}

\begin{proofofthmtorsion}
Let $P$ be a point on $S$. If $P$ is contained in the ramification curve of the morphism induced by the linear system of the bi-anticanonical divisor, then $P_{\E}$ is torsion (Remark \ref{remramcurve}), and we are done. Now assume that $P$ is not contained in this ramification curve, and that there is a set of at least 9 exceptional curves that are concurrent in $P$. Let $K_1,\ldots,K_n$ be the corresponding sections of~$\E$, and let $N$ be the height pairing matrix of these sections. Let $(a_1,\ldots,a_n)\in\mathbb{Z}^n$ be a vector in the kernel of $N$ such that $a_1+\cdots+a_n\neq0$, which exists by Lemma~\ref{rankandkernel}. Then for all $i\in\{1,\ldots,n\}$, we have that
$$a_1\langle K_i,K_1\rangle_h+\cdots+a_n\langle K_i,K_n\rangle_h=0,$$ and since the height pairing is bilinear this implies
\begin{equation}\langle K_i,a_1K_1+a_2K_2+\cdots+a_nK_n\rangle_h=0\mbox{ for all }i\in\{1,\ldots,n\},
\end{equation}
which implies $$\langle a_1K_1+a_2K_2+\cdots+a_nK_n,a_1K_1+a_2K_2+\cdots+a_nK_n\rangle_h=0.$$
From the latter we conclude that $a_1K_1+a_2K_2+\cdots+a_nK_n$ is torsion in the Mordell--Weil group of $\E$ \cite[Theorem 8.4]{Shioda_MW}, and since the torsion subgroup is trivial \cite[Theorem 10.4]{Shioda_MW}, we conclude that $$a_1K_1+a_2K_2+\cdots+a_nK_n=0.$$ Since for all $i$ in $\{1,\ldots,n\}$, the section $K_i$ contains the point $P_{\E}$, we have, on the fiber of $P_{\E}$, the equality $(a_1+\cdots+a_n)P_{\E}=0$. Since $a_1+\cdots+a_n\neq0$, this implies that $P_{\E}$ is torsion on its fiber. \qed\end{proofofthmtorsion}

\section{The case of $8$ lines}\label{Sec: 8lines}

Given that the rank of the Mordell-Weil group of the elliptic surface arising from a del Pezzo surface of degree 1 is 8 (see also Remark \ref{rem: rank 8 so 9 gives torsion}), we conjecture that points in the intersection of 8 lines are not always torsion, that is, we expect to find an example  of a del Pezzo surface with a non-torsion point contained in the intersection of 8 lines. However, we have not yet found such an example, nor do we have a proof that every point contained in the intersection of 8 lines is torsion on its fiber. In this section we prove Theorem \ref{thm: case8lines}: we show that 8 lines that are concurrent in a non-torsion point can only intersect each other according to one of 15 specific configurations, which reduces to 13 in characteristic 0. 
We also give strategies for searching for examples with non-torsion points within these cliques, or eliminating more cliques.

\vspace{11pt}

Let $S$ be a del Pezzo surface of degree 1 over an algebraically closed field, and let $\E$ be the associated elliptic surface. Let $G$ be the weigthed graph on the 240 exceptional curves as in Definition \ref{defgammaG}. Note that its automorphism group, the Weyl group $W_8$, acts on the set of cliques of size 8 with only edges of weights 1 and 2. 

\begin{proposition}\label{prop: 47 cliques of size 8}
There are $47$ orbits under the action of $W_8$ of cliques of size 8 with only edges of weights 1 and 2 in $G$. Their isomorphism types are represented in Figure \ref{table cliques}.
\end{proposition}
\begin{proof}Every clique in $G$ of size 8 with only edges of weights 1 and 2 contains at least one edge of weight 1, since the maximal size of cliques in $G$ with only edges of weight 2 is 3, which follows from \cite[Lemma 7]{WvL}, using the bijection (\ref{eq:bijection C E}). Using this same bijection, it follows from \cite[Proposition 6]{WvL} that $W_8$ acts transitively on the set of pairs of exceptional curves that intersect with multiplicity 1, so we fix two such curves; let $e_1$ be the strict transform on $S$ of the curve $L_{1,2}$ in $\mathbb{P}^2$, and $e_2$ the strict transform of $L_{3,4}$, where we use Notation \ref{notation curves P^2}. It follows that every clique of size 8 with only edges of weights 1 and 2 in $G$ is conjugate under the action of $W_8$ to a clique containing $e_1$ and $e_2$. With \texttt{magma} we compute that there are 136 exceptional curves that intersect both $e_1$ and $e_2$ with multiplicity 1 or 2. We define the graph $H$  with vertices these 136 exceptional curves, with an edge between two vertices if they correspond to exceptional curves intersecting with multiplicity 1 or 2, and no edge otherwise. The function \texttt{AllCliques(H,6,false)} in the \texttt{magma} code \cite{MagmaCode} gives all (not necessarily maximal) cliques of size 6 in this graph; there are 8963624 of them. We conclude that there are 8963624 cliques in $G$ of size 8 with only edges of weights 1 and 2, that contain $e_1$ and $e_2$. This set $A$ of cliques contains a representative for each $W_8$-orbit of cliques of size 8 with only edges of weights 1 and 2 in $G$, and we want to find a set of such representatives. To reduce computing time we first sort all cliques in $A$ according to the size of their stabilizer in $W_8$. This gives 20 different sets $A_1,\ldots,A_{20}$, where each set contains only cliques in $A$ with the same stabilizer size. Finally, for each of these sets $A_i$, we check whether the cliques inside are conjugate under the action of $W_8$, and end up with one representative for each $W_8$-orbit of the cliques in $A_i$. Doing this for all $A_i$ takes a very long time; we let \texttt{magma} run for two weeks straight on the compute servers of the Max Planck Institute for Mathematics in the Sciences, Leipzig. The output gave 47 cliques of size 8 with only edges of weight 1 and 2, where each clique is a representative for a different $W_8$-orbit. The isomorphism types follow from the pairwise intersection multiplicity of the exceptional curves.
\end{proof}

Figure \ref{table cliques} contains the isomorphism types of 47 cliques in $G$, one representative for each of the 47 orbits in Proposition \ref{prop: 47 cliques of size 8}; there are 45 different isomorphism types. All graphs are fully connected subgraphs of $G$ with edges of weights 2 (the ones that are drawn) and~1 (all other edges). 

\vspace{22pt}

\begin{center}
\begingroup
\begin{tabular}{ccccccccc}
\raisebox{1pt}{\begin{tikzpicture} [scale=0.3]
 \node [draw,circle,fill,inner sep=0pt,minimum size=2pt](t1) at (0,0) {};
 \node [draw,circle,fill,inner sep=0pt,minimum size=2pt](t2) at (1,0) {};
 \node [draw,circle,fill,inner sep=0pt,minimum size=2pt](t3) at (2,0) {};
 \node [draw,circle,fill,inner sep=0pt,minimum size=2pt](t3) at (3,0) {};
 \node [draw,circle,fill,inner sep=0pt,minimum size=2pt](t1) at (4,0) {};
 \node [draw,circle,fill,inner sep=0pt,minimum size=2pt](t2) at (5,0) {};
 \node [draw,circle,fill,inner sep=0pt,minimum size=2pt](t3) at (6,0) {};
 \node [draw,circle,fill,inner sep=0pt,minimum size=2pt](t3) at (7,0) {};
 \path[every node/.style={font=\sffamily\small}]
	(0,0) edge node {} (1,0)     
     (1,0) edge node {} (2,0)
     (2,0) edge node {} (3,0)
     (3,0) edge node {} (4,0)
     (4,0) edge node {} (5,0)
     (5,0) edge node {} (6,0)
     (6,0) edge node {} (7,0);
 \end{tikzpicture}}
 &&
 \raisebox{1pt}{\begin{tikzpicture} [scale=0.3]
 \node [draw,circle,fill,inner sep=0pt,minimum size=2pt](t1) at (0,0) {};
 \node [draw,circle,fill,inner sep=0pt,minimum size=2pt](t2) at (1,0) {};
 \node [draw,circle,fill,inner sep=0pt,minimum size=2pt](t3) at (2,0) {};
 \node [draw,circle,fill,inner sep=0pt,minimum size=2pt](t3) at (3,0) {};
 \node [draw,circle,fill,inner sep=0pt,minimum size=2pt](t1) at (4,0) {};
 \node [draw,circle,fill,inner sep=0pt,minimum size=2pt](t2) at (5,0) {};
 \node [draw,circle,fill,inner sep=0pt,minimum size=2pt](t3) at (6,0) {};
 \node [draw,circle,fill,inner sep=0pt,minimum size=2pt](t3) at (7,0) {};
 \path[every node/.style={font=\sffamily\small}]
	(0,0) edge node {} (1,0)     
     (1,0) edge node {} (2,0)
     (2,0) edge node {} (3,0)
     (3,0) edge node {} (4,0)
     (4,0) edge node {} (5,0)
     (5,0) edge node {} (6,0);
 \end{tikzpicture}}
 &&
 \raisebox{1pt}{\begin{tikzpicture} [scale=0.3]
 \node [draw,circle,fill,inner sep=0pt,minimum size=2pt](t1) at (0,0) {};
 \node [draw,circle,fill,inner sep=0pt,minimum size=2pt](t2) at (1,0) {};
 \node [draw,circle,fill,inner sep=0pt,minimum size=2pt](t3) at (2,0) {};
 \node [draw,circle,fill,inner sep=0pt,minimum size=2pt](t3) at (3,0) {};
 \node [draw,circle,fill,inner sep=0pt,minimum size=2pt](t1) at (4,0) {};
 \node [draw,circle,fill,inner sep=0pt,minimum size=2pt](t2) at (5,0) {};
 \node [draw,circle,fill,inner sep=0pt,minimum size=2pt](t3) at (6,0) {};
 \node [draw,circle,fill,inner sep=0pt,minimum size=2pt](t3) at (7,0) {};
 \path[every node/.style={font=\sffamily\small}]
	(0,0) edge node {} (1,0)     
     (1,0) edge node {} (2,0)
     (2,0) edge node {} (3,0)
     (6,0) edge node {} (7,0)
     (4,0) edge node {} (5,0)
     (5,0) edge node {} (6,0);
 \end{tikzpicture}}
&&
\raisebox{1pt}{\begin{tikzpicture} [scale=0.3]
 \node [draw,circle,fill,inner sep=0pt,minimum size=2pt](t1) at (0,0) {};
 \node [draw,circle,fill,inner sep=0pt,minimum size=2pt](t2) at (1,0) {};
 \node [draw,circle,fill,inner sep=0pt,minimum size=2pt](t3) at (2,0) {};
 \node [draw,circle,fill,inner sep=0pt,minimum size=2pt](t3) at (3,0) {};
 \node [draw,circle,fill,inner sep=0pt,minimum size=2pt](t1) at (4,0) {};
 \node [draw,circle,fill,inner sep=0pt,minimum size=2pt](t2) at (5,0) {};
 \node [draw,circle,fill,inner sep=0pt,minimum size=2pt](t3) at (6,0) {};
 \node [draw,circle,fill,inner sep=0pt,minimum size=2pt](t3) at (7,0) {};
 \path[every node/.style={font=\sffamily\small}]
	(0,0) edge node {} (1,0)     
     (1,0) edge node {} (2,0)
     (2,0) edge node {} (3,0)
     (4,0) edge node {} (5,0)
     (5,0) edge node {} (6,0);
 \end{tikzpicture}}
 &&
  \raisebox{1pt}{\begin{tikzpicture} [scale=0.3]
 \node [draw,circle,fill,inner sep=0pt,minimum size=2pt](t1) at (0,0) {};
 \node [draw,circle,fill,inner sep=0pt,minimum size=2pt](t2) at (1,0) {};
 \node [draw,circle,fill,inner sep=0pt,minimum size=2pt](t3) at (2,0) {};
 \node [draw,circle,fill,inner sep=0pt,minimum size=2pt](t3) at (3,0) {};
 \node [draw,circle,fill,inner sep=0pt,minimum size=2pt](t1) at (4,0) {};
 \node [draw,circle,fill,inner sep=0pt,minimum size=2pt](t2) at (5,0) {};
 \node [draw,circle,fill,inner sep=0pt,minimum size=2pt](t3) at (6,0) {};
 \node [draw,circle,fill,inner sep=0pt,minimum size=2pt](t3) at (7,0) {};
 \path[every node/.style={font=\sffamily\small}]
	(0,0) edge node {} (1,0)     
     (1,0) edge node {} (2,0)
     (2,0) edge node {} (3,0)
     (3,0) edge node {} (4,0)
     (5,0) edge node {} (6,0);
 \end{tikzpicture}} \\
1&&2&&3&&4&&5\\
&&&&&&&&\\
\raisebox{1pt}{\begin{tikzpicture} [scale=0.3]
 \node [draw,circle,fill,inner sep=0pt,minimum size=2pt](t1) at (0,0) {};
 \node [draw,circle,fill,inner sep=0pt,minimum size=2pt](t2) at (1,0) {};
 \node [draw,circle,fill,inner sep=0pt,minimum size=2pt](t3) at (2,0) {};
 \node [draw,circle,fill,inner sep=0pt,minimum size=2pt](t4) at (3,0) {};
 \node [draw,circle,fill,inner sep=0pt,minimum size=2pt](t5) at (4,0) {};
 \node [draw,circle,fill,inner sep=0pt,minimum size=2pt](t6) at (5,0) {};
 \node [draw,circle,fill,inner sep=0pt,minimum size=2pt](t7) at (6,0) {};
 \node [draw,circle,fill,inner sep=0pt,minimum size=2pt](t8) at (7,0) {};
 \path[every node/.style={font=\sffamily\small}]
	(0,0) edge node {} (1,0)     
     (2,0) edge node {} (3,0)
     (4,0) edge node {} (5,0)
     (6,0) edge node {} (7,0);
 \end{tikzpicture}}
&&
\raisebox{3pt}{\begin{tikzpicture} [scale=0.3]
 \node [draw,circle,fill,inner sep=0pt,minimum size=2pt](t1) at (0,0) {};
 \node [draw,circle,fill,inner sep=0pt,minimum size=2pt](t2) at (1,0) {};
 \node [draw,circle,fill,inner sep=0pt,minimum size=2pt](t3) at (2,0) {};
 \node [draw,circle,fill,inner sep=0pt,minimum size=2pt](t4) at (3,0) {};
 \node [draw,circle,fill,inner sep=0pt,minimum size=2pt](t5) at (4,0) {};
 \node [draw,circle,fill,inner sep=0pt,minimum size=2pt](t6) at (5,0) {};
 \node [draw,circle,fill,inner sep=0pt,minimum size=2pt](t7) at (6,0) {};
 \node [draw,circle,fill,inner sep=0pt,minimum size=2pt](t8) at (7,0) {};
 \end{tikzpicture}}
&&
\raisebox{1pt}{\begin{tikzpicture} [scale=0.3]
 \node [draw,circle,fill,inner sep=0pt,minimum size=2pt](t1) at (0,0) {};
 \node [draw,circle,fill,inner sep=0pt,minimum size=2pt](t2) at (1,0) {};
 \node [draw,circle,fill,inner sep=0pt,minimum size=2pt](t3) at (2,0) {};
 \node [draw,circle,fill,inner sep=0pt,minimum size=2pt](t4) at (3,0) {};
 \node [draw,circle,fill,inner sep=0pt,minimum size=2pt](t5) at (4,0) {};
 \node [draw,circle,fill,inner sep=0pt,minimum size=2pt](t6) at (5,0) {};
 \node [draw,circle,fill,inner sep=0pt,minimum size=2pt](t7) at (6,0) {};
 \node [draw,circle,fill,inner sep=0pt,minimum size=2pt](t8) at (1,-1) {};
 \path[every node/.style={font=\sffamily\small}]
	(0,0) edge node {} (1,0)     
     (1,0) edge node {} (2,0)
     (2,0) edge node {} (3,0)
     (3,0) edge node {} (4,0)
     (4,0) edge node {} (5,0)
     (1,0) edge node {} (1,-1)
     (5,0) edge node {} (6,0);
 \end{tikzpicture}}
 &&
 
 \raisebox{1pt}{\begin{tikzpicture} [scale=0.3]
 \node [draw,circle,fill,inner sep=0pt,minimum size=2pt](t1) at (0,0) {};
 \node [draw,circle,fill,inner sep=0pt,minimum size=2pt](t2) at (1,0) {};
 \node [draw,circle,fill,inner sep=0pt,minimum size=2pt](t3) at (2,0) {};
 \node [draw,circle,fill,inner sep=0pt,minimum size=2pt](t3) at (3,0) {};
 \node [draw,circle,fill,inner sep=0pt,minimum size=2pt](t1) at (4,0) {};
 \node [draw,circle,fill,inner sep=0pt,minimum size=2pt](t2) at (5,0) {};
 \node [draw,circle,fill,inner sep=0pt,minimum size=2pt](t3) at (6,0) {};
 \node [draw,circle,fill,inner sep=0pt,minimum size=2pt](t3) at (3,-1) {};
 \path[every node/.style={font=\sffamily\small}]
	(0,0) edge node {} (1,0)     
     (1,0) edge node {} (2,0)
     (2,0) edge node {} (3,0)
     (3,0) edge node {} (4,0)
     (4,0) edge node {} (5,0)
     (3,0) edge node {} (3,-1)
     (5,0) edge node {} (6,0);
 \end{tikzpicture}} 
 &&
 
\raisebox{1pt}{\begin{tikzpicture} [scale=0.3]
 \node [draw,circle,fill,inner sep=0pt,minimum size=2pt](t1) at (0,0) {};
 \node [draw,circle,fill,inner sep=0pt,minimum size=2pt](t2) at (1,0) {};
 \node [draw,circle,fill,inner sep=0pt,minimum size=2pt](t3) at (2,0) {};
 \node [draw,circle,fill,inner sep=0pt,minimum size=2pt](t4) at (3,0) {};
 \node [draw,circle,fill,inner sep=0pt,minimum size=2pt](t5) at (4,0) {};
 \node [draw,circle,fill,inner sep=0pt,minimum size=2pt](t6) at (5,0) {};
 \node [draw,circle,fill,inner sep=0pt,minimum size=2pt](t7) at (6,0) {};
 \node [draw,circle,fill,inner sep=0pt,minimum size=2pt](t8) at (4,-1) {};
 \path[every node/.style={font=\sffamily\small}]
	(0,0) edge node {} (1,0)     
     (1,0) edge node {} (2,0)
     (2,0) edge node {} (3,0)
     (3,0) edge node {} (4,0)
     (4,0) edge node {} (5,0)
     (4,0) edge node {} (4,-1)
     (5,0) edge node {} (6,0);
 \end{tikzpicture}}\\
6&&7&&8&&9&&10\\
&&&&&&&&\\
\raisebox{1pt}{\begin{tikzpicture} [scale=0.3]
 \node [draw,circle,fill,inner sep=0pt,minimum size=2pt](t1) at (0,0) {};
 \node [draw,circle,fill,inner sep=0pt,minimum size=2pt](t2) at (1,0) {};
 \node [draw,circle,fill,inner sep=0pt,minimum size=2pt](t3) at (2,0) {};
 \node [draw,circle,fill,inner sep=0pt,minimum size=2pt](t4) at (3,0) {};
 \node [draw,circle,fill,inner sep=0pt,minimum size=2pt](t5) at (4,0) {};
 \node [draw,circle,fill,inner sep=0pt,minimum size=2pt](t6) at (5,0) {};
 \node [draw,circle,fill,inner sep=0pt,minimum size=2pt](t7) at (6,0) {};
 \node [draw,circle,fill,inner sep=0pt,minimum size=2pt](t8) at (1,-1) {};
 \path[every node/.style={font=\sffamily\small}]
	(0,0) edge node {} (1,0)     
     (1,0) edge node {} (2,0)
     (2,0) edge node {} (3,0)
     (4,0) edge node {} (5,0)
     (1,0) edge node {} (1,-1)
     (5,0) edge node {} (6,0);
 \end{tikzpicture}}
&&
\raisebox{1pt}{\begin{tikzpicture} [scale=0.3]
 \node [draw,circle,fill,inner sep=0pt,minimum size=2pt](t1) at (0,0) {};
 \node [draw,circle,fill,inner sep=0pt,minimum size=2pt](t2) at (1,0) {};
 \node [draw,circle,fill,inner sep=0pt,minimum size=2pt](t3) at (2,0) {};
 \node [draw,circle,fill,inner sep=0pt,minimum size=2pt](t4) at (3,0) {};
 \node [draw,circle,fill,inner sep=0pt,minimum size=2pt](t5) at (4,0) {};
 \node [draw,circle,fill,inner sep=0pt,minimum size=2pt](t6) at (5,0) {};
 \node [draw,circle,fill,inner sep=0pt,minimum size=2pt](t7) at (6,0) {};
 \node [draw,circle,fill,inner sep=0pt,minimum size=2pt](t8) at (2,-1) {};
 \path[every node/.style={font=\sffamily\small}]
	(0,0) edge node {} (1,0)     
     (1,0) edge node {} (2,0)
     (2,0) edge node {} (3,0)
     (3,0) edge node {} (4,0)
     (2,0) edge node {} (2,-1)
     (5,0) edge node {} (6,0);
 \end{tikzpicture}}
&&
 \raisebox{1pt}{\begin{tikzpicture} [scale=0.3]
 \node [draw,circle,fill,inner sep=0pt,minimum size=2pt](t1) at (0,0) {};
 \node [draw,circle,fill,inner sep=0pt,minimum size=2pt](t2) at (1,0) {};
 \node [draw,circle,fill,inner sep=0pt,minimum size=2pt](t3) at (2,0) {};
 \node [draw,circle,fill,inner sep=0pt,minimum size=2pt](t3) at (3,0) {};
 \node [draw,circle,fill,inner sep=0pt,minimum size=2pt](t1) at (4,0) {};
 \node [draw,circle,fill,inner sep=0pt,minimum size=2pt](t2) at (5,0) {};
 \node [draw,circle,fill,inner sep=0pt,minimum size=2pt](t3) at (6,0) {};
 \node [draw,circle,fill,inner sep=0pt,minimum size=2pt](t3) at (3,-1) {};
 \path[every node/.style={font=\sffamily\small}]
	(0,0) edge node {} (1,0)     
     (1,0) edge node {} (2,0)
     (2,0) edge node {} (3,0)
     (3,0) edge node {} (4,0)
     (4,0) edge node {} (5,0)
     (3,0) edge node {} (3,-1);
 \end{tikzpicture}}
 &&
\raisebox{1pt}{\begin{tikzpicture} [scale=0.3]
 \node [draw,circle,fill,inner sep=0pt,minimum size=2pt](t1) at (0,0) {};
 \node [draw,circle,fill,inner sep=0pt,minimum size=2pt](t2) at (1,0) {};
 \node [draw,circle,fill,inner sep=0pt,minimum size=2pt](t3) at (2,0) {};
 \node [draw,circle,fill,inner sep=0pt,minimum size=2pt](t3) at (3,0) {};
 \node [draw,circle,fill,inner sep=0pt,minimum size=2pt](t1) at (4,0) {};
 \node [draw,circle,fill,inner sep=0pt,minimum size=2pt](t2) at (5,0) {};
 \node [draw,circle,fill,inner sep=0pt,minimum size=2pt](t3) at (6,0) {};
 \node [draw,circle,fill,inner sep=0pt,minimum size=2pt](t3) at (3,-1) {};
 \path[every node/.style={font=\sffamily\small}]
	(0,0) edge node {} (1,0)     
     (1,0) edge node {} (2,0)
     (2,0) edge node {} (3,0)
     (3,0) edge node {} (4,0)
     (3,0) edge node {} (3,-1);
 \end{tikzpicture}} 
&&
\raisebox{1pt}{\begin{tikzpicture} [scale=0.3]
 \node [draw,circle,fill,inner sep=0pt,minimum size=2pt](t1) at (0,0) {};
 \node [draw,circle,fill,inner sep=0pt,minimum size=2pt](t2) at (1,0) {};
 \node [draw,circle,fill,inner sep=0pt,minimum size=2pt](t3) at (2,0) {};
 \node [draw,circle,fill,inner sep=0pt,minimum size=2pt](t3) at (3,0) {};
 \node [draw,circle,fill,inner sep=0pt,minimum size=2pt](t1) at (4,0) {};
 \node [draw,circle,fill,inner sep=0pt,minimum size=2pt](t2) at (5,0) {};
 \node [draw,circle,fill,inner sep=0pt,minimum size=2pt](t3) at (6,0) {};
 \node [draw,circle,fill,inner sep=0pt,minimum size=2pt](t3) at (1,-1) {};
 \path[every node/.style={font=\sffamily\small}]
	(0,0) edge node {} (1,0)     
     (1,0) edge node {} (2,0)
     (1,0) edge node {} (1,-1);
 \end{tikzpicture}} \\
11&&12&&13&&14&&15\\
&&&&&&&&\\
 \raisebox{1pt}{\begin{tikzpicture} [scale=0.3]
 \node [draw,circle,fill,inner sep=0pt,minimum size=2pt](t1) at (0,0) {};
 \node [draw,circle,fill,inner sep=0pt,minimum size=2pt](t2) at (1,0) {};
 \node [draw,circle,fill,inner sep=0pt,minimum size=2pt](t3) at (2,0) {};
 \node [draw,circle,fill,inner sep=0pt,minimum size=2pt](t3) at (3,0) {};
 \node [draw,circle,fill,inner sep=0pt,minimum size=2pt](t1) at (4,0) {};
 \node [draw,circle,fill,inner sep=0pt,minimum size=2pt](t2) at (5,0) {};
 \node [draw,circle,fill,inner sep=0pt,minimum size=2pt](t3) at (1,-1) {};
 \node [draw,circle,fill,inner sep=0pt,minimum size=2pt](t3) at (4,-1) {};
 \path[every node/.style={font=\sffamily\small}]
	(0,0) edge node {} (1,0)     
     (1,0) edge node {} (2,0)
     (2,0) edge node {} (3,0)
     (1,0) edge node {} (1,-1)
     (4,0) edge node {} (5,0)
     (3,0) edge node {} (4,0)
     (4,0) edge node {} (4,-1);
 \end{tikzpicture}}
&&
\raisebox{1pt}{\begin{tikzpicture} [scale=0.3]
 \node [draw,circle,fill,inner sep=0pt,minimum size=2pt](t1) at (0,0) {};
 \node [draw,circle,fill,inner sep=0pt,minimum size=2pt](t2) at (1,0) {};
 \node [draw,circle,fill,inner sep=0pt,minimum size=2pt](t3) at (2,0) {};
 \node [draw,circle,fill,inner sep=0pt,minimum size=2pt](t4) at (3,0) {};
 \node [draw,circle,fill,inner sep=0pt,minimum size=2pt](t5) at (4,0) {};
 \node [draw,circle,fill,inner sep=0pt,minimum size=2pt](t6) at (5,0) {};
 \node [draw,circle,fill,inner sep=0pt,minimum size=2pt](t7) at (1,-1) {};
 \node [draw,circle,fill,inner sep=0pt,minimum size=2pt](t8) at (3,-1) {};
 \path[every node/.style={font=\sffamily\small}]
	(0,0) edge node {} (1,0)     
     (1,0) edge node {} (2,0)
     (2,0) edge node {} (3,0)
     (3,0) edge node {} (4,0)
     (1,0) edge node {} (1,-1)
     (3,0) edge node {} (3,-1);
 \end{tikzpicture}}
 &&
 \raisebox{1pt}{\begin{tikzpicture} [scale=0.3]
 \node [draw,circle,fill,inner sep=0pt,minimum size=2pt](t1) at (0,0) {};
 \node [draw,circle,fill,inner sep=0pt,minimum size=2pt](t2) at (1,0) {};
 \node [draw,circle,fill,inner sep=0pt,minimum size=2pt](t3) at (2,0) {};
 \node [draw,circle,fill,inner sep=0pt,minimum size=2pt](t3) at (3,0) {};
 \node [draw,circle,fill,inner sep=0pt,minimum size=2pt](t1) at (4,0) {};
 \node [draw,circle,fill,inner sep=0pt,minimum size=2pt](t2) at (5,0) {};
 \node [draw,circle,fill,inner sep=0pt,minimum size=2pt](t3) at (1,-1) {};
 \node [draw,circle,fill,inner sep=0pt,minimum size=2pt](t3) at (2,-1) {};
 \path[every node/.style={font=\sffamily\small}]
	(0,0) edge node {} (1,0)     
     (1,0) edge node {} (2,0)
     (2,0) edge node {} (3,0)
     (1,0) edge node {} (1,-1)
     (2,0) edge node {} (2,-1)
     (4,0) edge node {} (5,0);
 \end{tikzpicture}} 
&&
\raisebox{1pt}{\begin{tikzpicture} [scale=0.3]
 \node [draw,circle,fill,inner sep=0pt,minimum size=2pt](t1) at (0,0) {};
 \node [draw,circle,fill,inner sep=0pt,minimum size=2pt](t2) at (1,0) {};
 \node [draw,circle,fill,inner sep=0pt,minimum size=2pt](t3) at (2,0) {};
 \node [draw,circle,fill,inner sep=0pt,minimum size=2pt](t4) at (3,0) {};
 \node [draw,circle,fill,inner sep=0pt,minimum size=2pt](t5) at (4,0) {};
 \node [draw,circle,fill,inner sep=0pt,minimum size=2pt](t6) at (5,0) {};
 \node [draw,circle,fill,inner sep=0pt,minimum size=2pt](t7) at (2,-1) {};
 \node [draw,circle,fill,inner sep=0pt,minimum size=2pt](t8) at (4,-1) {};
 \path[every node/.style={font=\sffamily\small}]
	(0,0) edge node {} (1,0)     
     (1,0) edge node {} (2,0)
     (2,0) edge node {} (2,-1)
     (3,0) edge node {} (4,0)
     (4,0) edge node {} (5,0)
     (4,0) edge node {} (4,-1);
 \end{tikzpicture}}
&&
\raisebox{1pt}{\begin{tikzpicture} [scale=0.3]
 \node [draw,circle,fill,inner sep=0pt,minimum size=2pt](t1) at (0,0) {};
 \node [draw,circle,fill,inner sep=0pt,minimum size=2pt](t2) at (1,0) {};
 \node [draw,circle,fill,inner sep=0pt,minimum size=2pt](t3) at (2,0) {};
 \node [draw,circle,fill,inner sep=0pt,minimum size=2pt](t4) at (3,0) {};
 \node [draw,circle,fill,inner sep=0pt,minimum size=2pt](t5) at (4,0) {};
 \node [draw,circle,fill,inner sep=0pt,minimum size=2pt](t6) at (5,0) {};
 \node [draw,circle,fill,inner sep=0pt,minimum size=2pt](t7) at (1,-1) {};
 \node [draw,circle,fill,inner sep=0pt,minimum size=2pt](t8) at (2,-1) {};
 \path[every node/.style={font=\sffamily\small}]
	(0,0) edge node {} (1,0)     
     (1,0) edge node {} (2,0)
     (2,0) edge node {} (3,0)
     (1,0) edge node {} (1,-1)
     (2,0) edge node {} (2,-1);
 \end{tikzpicture}}\\
16&&17&&18&&19&&20\\
&&&&&&&&\\
\raisebox{1pt}{\begin{tikzpicture} [scale=0.3]
 \node [draw,circle,fill,inner sep=0pt,minimum size=2pt](t1) at (0,0) {};
 \node [draw,circle,fill,inner sep=0pt,minimum size=2pt](t2) at (1,0) {};
 \node [draw,circle,fill,inner sep=0pt,minimum size=2pt](t3) at (2,0) {};
 \node [draw,circle,fill,inner sep=0pt,minimum size=2pt](t4) at (3,0) {};
 \node [draw,circle,fill,inner sep=0pt,minimum size=2pt](t5) at (4,0) {};
 \node [draw,circle,fill,inner sep=0pt,minimum size=2pt](t6) at (5,0) {};
 \node [draw,circle,fill,inner sep=0pt,minimum size=2pt](t7) at (2,-1) {};
 \node [draw,circle,fill,inner sep=0pt,minimum size=2pt](t8) at (2,-2) {};
 \path[every node/.style={font=\sffamily\small}]
	(0,0) edge node {} (1,0)     
     (1,0) edge node {} (2,0)
     (2,0) edge node {} (3,0)
     (3,0) edge node {} (4,0)
     (2,0) edge node {} (2,-1)
     (2,-1) edge node {} (2,-2);
 \end{tikzpicture}}
&&
\raisebox{1pt}{\begin{tikzpicture} [scale=0.3]
 \node [draw,circle,fill,inner sep=0pt,minimum size=2pt](t1) at (0,0) {};
 \node [draw,circle,fill,inner sep=0pt,minimum size=2pt](t2) at (1,0) {};
 \node [draw,circle,fill,inner sep=0pt,minimum size=2pt](t3) at (2,0) {};
 \node [draw,circle,fill,inner sep=0pt,minimum size=2pt](t3) at (3,0) {};
 \node [draw,circle,fill,inner sep=0pt,minimum size=2pt](t1) at (4,0) {};
 \node [draw,circle,fill,inner sep=0pt,minimum size=2pt](t2) at (5,0) {};
 \node [draw,circle,fill,inner sep=0pt,minimum size=2pt](t3) at (1,1) {};
 \node [draw,circle,fill,inner sep=0pt,minimum size=2pt](t3) at (1,-1) {};
 \path[every node/.style={font=\sffamily\small}]
	(0,0) edge node {} (1,0)     
     (1,0) edge node {} (2,0)
     (1,0) edge node {} (1,-1)
     (1,0) edge node {} (1,1);
 \end{tikzpicture}} 
&&
\raisebox{1pt}{\begin{tikzpicture} [scale=0.3]
 \node [draw,circle,fill,inner sep=0pt,minimum size=2pt](t1) at (0,0) {};
 \node [draw,circle,fill,inner sep=0pt,minimum size=2pt](t2) at (1,0) {};
 \node [draw,circle,fill,inner sep=0pt,minimum size=2pt](t3) at (2,0) {};
 \node [draw,circle,fill,inner sep=0pt,minimum size=2pt](t3) at (3,0) {};
 \node [draw,circle,fill,inner sep=0pt,minimum size=2pt](t1) at (4,0) {};
 \node [draw,circle,fill,inner sep=0pt,minimum size=2pt](t2) at (1,1) {};
 \node [draw,circle,fill,inner sep=0pt,minimum size=2pt](t3) at (1,-1) {};
 \node [draw,circle,fill,inner sep=0pt,minimum size=2pt](t3) at (3,-1) {};
 \path[every node/.style={font=\sffamily\small}]
	(0,0) edge node {} (1,0)     
     (1,0) edge node {} (2,0)
     (2,0) edge node {} (3,0)
     (3,0) edge node {} (4,0)
     (1,0) edge node {} (1,1)
     (1,0) edge node {} (1,-1)
     (3,0) edge node {} (3,-1);
 \end{tikzpicture}}
 &&
 &&
 \\
21&&22&&23&&&&\\
&&&&&&&&\\
 \raisebox{1pt}{\begin{tikzpicture} [scale=0.3]
 \node [draw,circle,fill,inner sep=0pt,minimum size=2pt](t1) at (0,0) {};
 \node [draw,circle,fill,inner sep=0pt,minimum size=2pt](t2) at (1,0) {};
 \node [draw,circle,fill,inner sep=0pt,minimum size=2pt](t3) at (2,0) {};
 \node [draw,circle,fill,inner sep=0pt,minimum size=2pt](t3) at (3,0) {};
 \node [draw,circle,fill,inner sep=0pt,minimum size=2pt](t3) at (4,0) {}; 
 \foreach \x /\alph  in {90/a, 210/b, 330/c}{
  \node[circle,fill,inner sep=0pt,minimum size=2pt,draw,xshift=2cm] (\alph) at (\x:0.8cm) {}; }
  \path[every node/.style={font=\sffamily\small}]
	(0,0) edge node {} (1,0)     
     (1,0) edge node {} (2,0)
     (2,0) edge node {} (3,0)
     (3,0) edge node {} (4,0)
     (a) edge node {} (b)
     (a) edge node {} (c)
     (c) edge node {} (b);
    \end{tikzpicture}}
    &&
 \raisebox{1pt}{\begin{tikzpicture} [scale=0.3]
 \node [draw,circle,fill,inner sep=0pt,minimum size=2pt](t1) at (0,0) {};
 \node [draw,circle,fill,inner sep=0pt,minimum size=2pt](t2) at (1,0) {};
 \node [draw,circle,fill,inner sep=0pt,minimum size=2pt](t3) at (2,0) {};
 \node [draw,circle,fill,inner sep=0pt,minimum size=2pt](t3) at (3,0) {};
 \node [draw,circle,fill,inner sep=0pt,minimum size=2pt](t3) at (4,0) {}; 
 \foreach \x /\alph  in {90/a, 210/b, 330/c}{
  \node[circle,fill,inner sep=0pt,minimum size=2pt,draw,xshift=2cm] (\alph) at (\x:0.8cm) {}; }
  \path[every node/.style={font=\sffamily\small}]
	(0,0) edge node {} (1,0)     
     (1,0) edge node {} (2,0)
     (2,0) edge node {} (3,0)
     (a) edge node {} (b)
     (a) edge node {} (c)
     (c) edge node {} (b);
    \end{tikzpicture}}
&&
\raisebox{1pt}{\begin{tikzpicture} [scale=0.3]
 \node [draw,circle,fill,inner sep=0pt,minimum size=2pt](t1) at (0,0) {};
 \node [draw,circle,fill,inner sep=0pt,minimum size=2pt](t2) at (1,0) {};
 \node [draw,circle,fill,inner sep=0pt,minimum size=2pt](t3) at (2,0) {};
 \node [draw,circle,fill,inner sep=0pt,minimum size=2pt](t3) at (3,0) {};
 \node [draw,circle,fill,inner sep=0pt,minimum size=2pt](t3) at (4,0) {}; 
 \foreach \x /\alph  in {90/a, 210/b, 330/c}{
  \node[circle,fill,inner sep=0pt,minimum size=2pt,draw,xshift=2cm] (\alph) at (\x:0.8cm) {}; }
  \path[every node/.style={font=\sffamily\small}]
	(0,0) edge node {} (1,0)     
     (1,0) edge node {} (2,0)
     (a) edge node {} (b)
     (a) edge node {} (c)
     (c) edge node {} (b);
    \end{tikzpicture}}
&&
 \raisebox{1pt}{\begin{tikzpicture} [scale=0.3]
 \node [draw,circle,fill,inner sep=0pt,minimum size=2pt](t1) at (0,0) {};
 \node [draw,circle,fill,inner sep=0pt,minimum size=2pt](t2) at (1,0) {};
 \node [draw,circle,fill,inner sep=0pt,minimum size=2pt](t3) at (2,0) {};
 \node [draw,circle,fill,inner sep=0pt,minimum size=2pt](t3) at (3,0) {};
 \node [draw,circle,fill,inner sep=0pt,minimum size=2pt](t3) at (4,0) {}; 
 \foreach \x /\alph  in {90/a, 210/b, 330/c}{
  \node[circle,fill,inner sep=0pt,minimum size=2pt,draw,xshift=2cm] (\alph) at (\x:0.8cm) {}; }
  \path[every node/.style={font=\sffamily\small}]
	(0,0) edge node {} (1,0)     
     (2,0) edge node {} (3,0)
     (a) edge node {} (b)
     (a) edge node {} (c)
     (c) edge node {} (b);
    \end{tikzpicture}}
&& 
\raisebox{1pt}{\begin{tikzpicture} [scale=0.3]
 \node [draw,circle,fill,inner sep=0pt,minimum size=2pt](t1) at (0,0) {};
 \node [draw,circle,fill,inner sep=0pt,minimum size=2pt](t2) at (1,0) {};
 \node [draw,circle,fill,inner sep=0pt,minimum size=2pt](t3) at (2,0) {};
 \node [draw,circle,fill,inner sep=0pt,minimum size=2pt](t3) at (3,0) {};
 \node [draw,circle,fill,inner sep=0pt,minimum size=2pt](t3) at (2,-1) {}; 
 \foreach \x /\alph  in {90/a, 210/b, 330/c}{
  \node[circle,fill,inner sep=0pt,minimum size=2pt,draw,xshift=1.5cm] (\alph) at (\x:0.8cm) {}; }
  \path[every node/.style={font=\sffamily\small}]
	(0,0) edge node {} (1,0)     
     (1,0) edge node {} (2,0)
     (2,0) edge node {} (3,0)
      (2,0) edge node {} (2,-1)
     (a) edge node {} (b)
     (a) edge node {} (c)
     (c) edge node {} (b);
    \end{tikzpicture}} \\
24&&25&&26&&27&&28\\
&&&&&&&&\\
\raisebox{1pt}{\begin{tikzpicture} [scale=0.3]
 \node [draw,circle,fill,inner sep=0pt,minimum size=2pt](t1) at (0,0) {};
 \node [draw,circle,fill,inner sep=0pt,minimum size=2pt](t2) at (1,0) {};
  \node [draw,circle,fill,inner sep=0pt,minimum size=2pt](t2) at (2,0) {};
 \node [draw,circle,fill,inner sep=0pt,minimum size=2pt](t3) at (1,1) {};
 \node [draw,circle,fill,inner sep=0pt,minimum size=2pt](t3) at (1,-1) {};
 \foreach \x /\alph  in {90/a, 210/b, 330/c}{
  \node[circle,fill,inner sep=0pt,minimum size=2pt,draw,xshift=1.5cm] (\alph) at (\x:0.8cm) {}; }
  \path[every node/.style={font=\sffamily\small}]
  (0,0) edge node {} (1,0)     
     (1,0) edge node {} (2,0)
     (1,0) edge node {} (1,-1)
     (1,0) edge node {} (1,1)
     (a) edge node {} (b)
     (a) edge node {} (c)
     (c) edge node {} (b);
    \end{tikzpicture}}
    &&
\raisebox{1pt}{\begin{tikzpicture} [scale=0.3]
 \node [draw,circle,fill,inner sep=0pt,minimum size=2pt](t1) at (0,0) {};
 \node [draw,circle,fill,inner sep=0pt,minimum size=2pt](t2) at (1,0) {};
 \node [draw,circle,fill,inner sep=0pt,minimum size=2pt](t3) at (2,0) {};
 \node [draw,circle,fill,inner sep=0pt,minimum size=2pt](t3) at (3,0) {}; 
 \foreach \x /\alph  in {90/a, 0/b, 180/c,  270/d}{
  \node[circle,fill,inner sep=0pt,minimum size=2pt,draw,xshift=1.5cm] (\alph) at (\x:0.8cm) {}; }
  \path[every node/.style={font=\sffamily\small}]
	(0,0) edge node {} (1,0)     
     (1,0) edge node {} (2,0)
     (2,0) edge node {} (3,0)
     (a) edge node {} (b)
     (a) edge node {} (c)
     (c) edge node {} (d)
     (d) edge node {} (b);
    \end{tikzpicture}}
 &&
 \raisebox{1pt}{\begin{tikzpicture} [scale=0.3]
 \node [draw,circle,fill,inner sep=0pt,minimum size=2pt](t1) at (0,0) {};
 \node [draw,circle,fill,inner sep=0pt,minimum size=2pt](t2) at (1,0) {};
 \node [draw,circle,fill,inner sep=0pt,minimum size=2pt](t3) at (2,0) {};
 \node [draw,circle,fill,inner sep=0pt,minimum size=2pt](t3) at (3,0) {}; 
 \foreach \x /\alph  in {90/a, 0/b, 180/c,  270/d}{
  \node[circle,fill,inner sep=0pt,minimum size=2pt,draw,xshift=1.5cm] (\alph) at (\x:0.8cm) {}; }
  \path[every node/.style={font=\sffamily\small}]
	(0,0) edge node {} (1,0)     
     (1,0) edge node {} (2,0)
     (a) edge node {} (b)
     (a) edge node {} (c)
     (c) edge node {} (d)
     (d) edge node {} (b);
    \end{tikzpicture}} 
 &&
\raisebox{1pt}{\begin{tikzpicture} [scale=0.3]
 \node [draw,circle,fill,inner sep=0pt,minimum size=2pt](t1) at (0,0) {};
 \node [draw,circle,fill,inner sep=0pt,minimum size=2pt](t2) at (1,0) {};
 \node [draw,circle,fill,inner sep=0pt,minimum size=2pt](t3) at (2,0) {};
 \node [draw,circle,fill,inner sep=0pt,minimum size=2pt](t3) at (3,0) {}; 
 \foreach \x /\alph  in {90/a, 0/b, 180/c,  270/d}{
  \node[circle,fill,inner sep=0pt,minimum size=2pt,draw,xshift=1.5cm] (\alph) at (\x:0.8cm) {}; }
  \path[every node/.style={font=\sffamily\small}]
	(0,0) edge node {} (1,0)     
     (a) edge node {} (b)
     (a) edge node {} (c)
     (c) edge node {} (d)
     (d) edge node {} (b);
    \end{tikzpicture}}
&&
\raisebox{1pt}{\begin{tikzpicture} [scale=0.3]
 \node [draw,circle,fill,inner sep=0pt,minimum size=2pt](t1) at (0,0) {};
 \node [draw,circle,fill,inner sep=0pt,minimum size=2pt](t2) at (1,0) {};
 \node [draw,circle,fill,inner sep=0pt,minimum size=2pt](t3) at (2,0) {};
 \node [draw,circle,fill,inner sep=0pt,minimum size=2pt](t3) at (3,0) {}; 
 \foreach \x /\alph  in {90/a, 0/b, 180/c,  270/d}{
  \node[circle,fill,inner sep=0pt,minimum size=2pt,draw,xshift=1.5cm] (\alph) at (\x:0.8cm) {}; }
  \path[every node/.style={font=\sffamily\small}]  
     (a) edge node {} (b)
     (a) edge node {} (c)
     (c) edge node {} (d)
     (d) edge node {} (b);
    \end{tikzpicture}}\\
29&&30&&31&&32&&33\\
&&&&&&&&\\
 \raisebox{1pt}{\begin{tikzpicture} [scale=0.3]
 \node [draw,circle,fill,inner sep=0pt,minimum size=2pt](t1) at (0,0) {};
 \node [draw,circle,fill,inner sep=0pt,minimum size=2pt](t2) at (1,0) {};
 \node [draw,circle,fill,inner sep=0pt,minimum size=2pt](t3) at (2,0) {};
 \node [draw,circle,fill,inner sep=0pt,minimum size=2pt](t3) at (1,-1) {}; 
 \foreach \x /\alph  in {90/a, 0/b, 180/c,  270/d}{
  \node[circle,fill,inner sep=0pt,minimum size=2pt,draw,xshift=1.5cm] (\alph) at (\x:0.8cm) {}; }
  \path[every node/.style={font=\sffamily\small}]
	(0,0) edge node {} (1,0)     
     (1,0) edge node {} (2,0)
     (1,0) edge node {} (1,-1)
     (a) edge node {} (b)
     (a) edge node {} (c)
     (c) edge node {} (d)
     (d) edge node {} (b);
    \end{tikzpicture}}
  &&
  \raisebox{1pt}{\begin{tikzpicture} [scale=0.3]
 \node [draw,circle,fill,inner sep=0pt,minimum size=2pt](t1) at (0,0) {};
 \node [draw,circle,fill,inner sep=0pt,minimum size=2pt](t2) at (1,0) {};
 \node [draw,circle,fill,inner sep=0pt,minimum size=2pt](t3) at (2,0) {};
 \foreach \x /\alph  in {18/a, 90/b, 162/c,  234/d, 309/e}{
  \node[circle,fill,inner sep=0pt,minimum size=2pt,draw,xshift=1.5cm] (\alph) at (\x:0.8cm) {}; }
 \path[every node/.style={font=\sffamily\small}]
	(0,0) edge node {} (1,0)
	(1,0) edge node {} (2,0)     
     (b) edge node {} (c)
     (b) edge node {} (a)
     (a) edge node {} (e)
     (d) edge node {} (e)
     (c) edge node {} (d);
 \end{tikzpicture}}
&&
 \raisebox{1pt}{\begin{tikzpicture} [scale=0.3]
 \node [draw,circle,fill,inner sep=0pt,minimum size=2pt](t1) at (0,0) {};
 \node [draw,circle,fill,inner sep=0pt,minimum size=2pt](t2) at (1,0) {};
 \node [draw,circle,fill,inner sep=0pt,minimum size=2pt](t3) at (2,0) {};
 \foreach \x /\alph  in {18/a, 90/b, 162/c,  234/d, 309/e}{
  \node[circle,fill,inner sep=0pt,minimum size=2pt,draw,xshift=1.5cm] (\alph) at (\x:0.8cm) {}; }
 \path[every node/.style={font=\sffamily\small}]
	(0,0) edge node {} (1,0)     
     (b) edge node {} (c)
     (b) edge node {} (a)
     (a) edge node {} (e)
     (d) edge node {} (e)
     (c) edge node {} (d);
 \end{tikzpicture}} 
 &&
 \raisebox{1pt}{\begin{tikzpicture} [scale=0.3]
 \node [draw,circle,fill,inner sep=0pt,minimum size=2pt](t1) at (0,0) {};
 \node [draw,circle,fill,inner sep=0pt,minimum size=2pt](t2) at (1,0) {};
 \foreach \x /\alph  in {90/a, 150/b, 210/c,  270/d, 330/e, 30/f}{
  \node[circle,fill,inner sep=0pt,minimum size=2pt,draw,xshift=1cm] (\alph) at (\x:0.8cm) {}; }
 \path[every node/.style={font=\sffamily\small}]    
 (0,0) edge node {} (1,0)   
     (a) edge node {} (b)
     (b) edge node {} (c)
     (c) edge node {} (d)
     (d) edge node {} (e)
     (e) edge node {} (f)
     (f) edge node {} (a);
 \end{tikzpicture}}
&&
\raisebox{1pt}{\begin{tikzpicture} [scale=0.3]
 \node [draw,circle,fill,inner sep=0pt,minimum size=2pt](t1) at (0,0) {};
 \node [draw,circle,fill,inner sep=0pt,minimum size=2pt](t2) at (1,0) {};
 \foreach \x /\alph  in {90/a, 150/b, 210/c,  270/d, 330/e, 30/f}{
  \node[circle,fill,inner sep=0pt,minimum size=2pt,draw,xshift=1cm] (\alph) at (\x:0.8cm) {}; }
 \path[every node/.style={font=\sffamily\small}]    
     (a) edge node {} (b)
     (b) edge node {} (c)
     (c) edge node {} (d)
     (d) edge node {} (e)
     (e) edge node {} (f)
     (f) edge node {} (a);
 \end{tikzpicture}}\\
34&&35&&36&&37&&38\\
&&&&&&&&\\
\raisebox{1pt}{\begin{tikzpicture} [scale=0.3]
 \node [draw,circle,fill,inner sep=0pt,minimum size=2pt](t1) at (0,0) {};
 \foreach \x /\alph  in {347/a, 41/b, 90/c,  141/d,  193/e,  244/f, 295/g}{
  \node[circle,fill,inner sep=0pt,minimum size=2pt,draw,xshift=1cm] (\alph) at (\x:0.8cm) {}; }
 \path[every node/.style={font=\sffamily\small}]    
     (a) edge node {} (b)
     (b) edge node {} (c)
     (c) edge node {} (d)
     (d) edge node {} (e)
     (e) edge node {} (f)
     (f) edge node {} (g)
     (g) edge node {} (a);
 \end{tikzpicture}} 
 &&
\raisebox{1pt}{\begin{tikzpicture} [scale=0.3]
 \foreach \x /\alph  in {22.5/a, 67.5/b, 112.5/c,  157.5/d,  202.5/e,  247.5/f, 292.5/g, 337.5/h}{
  \node[circle,fill,inner sep=0pt,minimum size=2pt,draw] (\alph) at (\x:0.8cm) {}; }
 \path[every node/.style={font=\sffamily\small}]    
     (a) edge node {} (b)
     (b) edge node {} (c)
     (c) edge node {} (d)
     (d) edge node {} (e)
     (e) edge node {} (f)
     (f) edge node {} (g)
     (g) edge node {} (h)
     (h) edge node {} (a);
 \end{tikzpicture}}
&&
\raisebox{1pt}{\begin{tikzpicture} [scale=0.3]
 \node [draw,circle,fill,inner sep=0pt,minimum size=2pt](t1) at (0,0) {};
 \node [draw,circle,fill,inner sep=0pt,minimum size=2pt](t2) at (1,0) {};
 \foreach \x /\alph  in {90/a, 210/b, 330/c}{
  \node[circle,fill,inner sep=0pt,minimum size=2pt,draw,xshift=1cm] (\alph) at (\x:0.8cm) {}; }
   \foreach \x /\alph  in {90/d, 210/e, 330/f}{
  \node[circle,fill,inner sep=0pt,minimum size=2pt,draw,xshift=2cm] (\alph) at (\x:0.8cm) {}; }
  \path[every node/.style={font=\sffamily\small}]
  (0,0) edge node {} (1,0)
     (a) edge node {} (b)
     (a) edge node {} (c)
     (c) edge node {} (b)
     (d) edge node {} (e)
     (e) edge node {} (f)
     (f) edge node {} (d);
    \end{tikzpicture}}
 &&
 \raisebox{1pt}{\begin{tikzpicture} [scale=0.3]
 \node [draw,circle,fill,inner sep=0pt,minimum size=2pt](t1) at (0,0) {};
 \node [draw,circle,fill,inner sep=0pt,minimum size=2pt](t2) at (1,0) {};
 \foreach \x /\alph  in {90/a, 210/b, 330/c}{
  \node[circle,fill,inner sep=0pt,minimum size=2pt,draw,xshift=2cm] (\alph) at (\x:0.8cm) {}; }
   \foreach \x /\alph  in {90/d, 210/e, 330/f}{
  \node[circle,fill,inner sep=0pt,minimum size=2pt,draw,xshift=1cm] (\alph) at (\x:0.8cm) {}; }
  \path[every node/.style={font=\sffamily\small}]
     (a) edge node {} (b)
     (a) edge node {} (c)
     (c) edge node {} (b)
     (d) edge node {} (e)
     (e) edge node {} (f)
     (f) edge node {} (d);
    \end{tikzpicture}}
    &&
 \raisebox{1pt}{\begin{tikzpicture} [scale=0.3]
 \foreach \x /\alph  in {90/a, 0/b, 180/c,  270/d}{
  \node[circle,fill,inner sep=0pt,minimum size=2pt,draw,xshift=0.75cm] (\alph) at (\x:0.8cm) {}; }
  \foreach \x /\alph  in {90/e, 0/f, 180/g,  270/h}{
  \node[circle,fill,inner sep=0pt,minimum size=2pt,draw,xshift=1.5cm] (\alph) at (\x:0.8cm) {}; }
  \path[every node/.style={font=\sffamily\small}]
     (a) edge node {} (b)
     (a) edge node {} (c)
     (c) edge node {} (d)
     (d) edge node {} (b)
     (e) edge node {} (f)
     (f) edge node {} (h)
     (g) edge node {} (h)
     (g) edge node {} (e);
    \end{tikzpicture}} \\
39&&40,41&&42&&43&&44,45\\
&&&&&&&&\\
 \raisebox{1pt}{\begin{tikzpicture} [scale=0.3]
 \node [draw,circle,fill,inner sep=0pt,minimum size=2pt](t1) at (0,0) {};
  \foreach \x /\alph  in {90/a, 210/b, 330/c}{
  \node[circle,fill,inner sep=0pt,minimum size=2pt,draw,xshift=0.75cm] (\alph) at (\x:0.8cm) {}; }
  \foreach \x /\alph  in {90/e, 0/f, 180/g,  270/h}{
  \node[circle,fill,inner sep=0pt,minimum size=2pt,draw,xshift=1.5cm] (\alph) at (\x:0.8cm) {}; }
  \path[every node/.style={font=\sffamily\small}]
     (a) edge node {} (b)
     (b) edge node {} (c)
     (c) edge node {} (a)
     (e) edge node {} (f)
     (f) edge node {} (h)
     (g) edge node {} (h)
     (g) edge node {} (e);
    \end{tikzpicture}} 
&&
 \raisebox{1pt}{\begin{tikzpicture} [scale=0.3]
 \foreach \x /\alph  in {90/f, 210/g, 330/h}{
  \node[circle,fill,inner sep=0pt,minimum size=2pt,draw] (\alph) at (\x:0.8cm) {}; }
   \foreach \x /\alph  in {18/a, 90/b, 162/c,  234/d, 309/e}{
  \node[circle,fill,inner sep=0pt,minimum size=2pt,draw,xshift=1cm] (\alph) at (\x:0.8cm) {}; }
 \path[every node/.style={font=\sffamily\small}]
     (b) edge node {} (c)
     (b) edge node {} (a)
     (a) edge node {} (e)
     (d) edge node {} (e)
     (c) edge node {} (d)
     (f) edge node {} (g)
     (g) edge node {} (h)
     (h) edge node {} (f);
    \end{tikzpicture}}

&&
  
 &&
  \\
46&&47
 \end{tabular}

\captionof{figure}{Isomorphism types of the cliques in 47 orbits of the set of cliques of size 8 with only edges of weights 1 and 2. All graphs are fully connected subgraphs with edges of weights 2 (the ones that are drawn) and~1 (all other edges). }
\label{table cliques}
\endgroup
\end{center}

\vspace{11pt}

\begin{definition}
We say that a clique in $G$ contains an \textsl{n-gon} if it contains a set of $n$ vertices $e_1,\ldots,e_n$, corresponding to $n$ exceptional curves $c_1,\ldots,c_n$ that intersect with multiplicity 2 for all $\{i,j\}\in\{\{a,a+1\}\colon a\in\{1,\ldots,n-1\}\}\cup\{\{1,n\}\}$, and with multiplicity 1 otherwise. 
\end{definition}

\begin{proposition}\label{prop: n-gon torsion}
If a clique in $G$ contains an $n$-gon corresponding to $n$ exceptional curves that are concurrent in a point $P\in S$, then the point $P_{\E}$ is torsion on its fiber, of order dividing $n$.  
\end{proposition}
\begin{proof}
Let $e_1,\ldots,e_n$ be the $n$ exceptional curves corresponding to an $n$-gon, and Let $L_1,\ldots,L_n$ be the corresponding sections of $\E$. From Lemma \ref{heightpairingisdotproduct} and the bijection in (\ref{eq:bijection C E}) it follows that, in the Mordell-Weil group of $\E$, we have $\left\langle L_i,L_i\right\rangle_h=2$ for all $i\in\{1,\ldots,n\}$, and $\left\langle L_i,L_j\right\rangle_h=-1$ for $\{i,j\}\in\{\{a,a+1\}\colon a\in\{1,\ldots,n-1\}\}\cup\{\{1,n\}\}$, and $\left\langle L_i,L_j\right\rangle_h=0$ otherwise. We find \begin{align*}
\left\langle L_1+\cdots+L_n,L_1+\cdots+L_n\right\rangle_h&=\sum_{i=1}^n\left\langle L_i,L_i\right\rangle_h+\sum_{i=1}^n\left\langle L_i,L_1+\cdots+L_{i-1}+L_{i+1}+\cdots+L_n\right\rangle_h\\
&=2n+n(-2)=0.\end{align*}
Therefore we have that $L_1+\cdots+L_n$ is torsion in the Mordell--Weil group of $\E$ \cite[Theorem 8.4]{Shioda_MW}, and since the torsion subgroup is trivial \cite[Theorem 10.4]{Shioda_MW}, we conclude that $$L_1+\cdots+L_n=0.$$ Specializing the section $L_1+\cdots+L_n$ to the fiber of $P_{\E}$, gives $nP_{\E}=0$. 
\end{proof}

From Proposition \ref{prop: 47 cliques of size 8} it follows that there is only one orbit of cliques of types 7 and 15 in Figure~\ref{table cliques}. Combining this with the following proposition, we can deduce that, in characteristic 0, cliques with those isomorphism types do not correspond to lines that are concurrent in a point on a del Pezzo surface. We do this in the proof of Theorem \ref{thm: case8lines}.

\begin{proposition}\label{propno3cubics}
Let $k$ be a field of characteristic 0. Let $P_1,\ldots,P_8$ be points in general position in $\mathbb{P}_k^2$, such that the lines $L_{1,2}$, $L_{3,4}$, $L_{5,6}$ and $L_{7,8}$ are concurrent in a point $P$. Then the curves in $\mathbb{P}_k^2$ given by $C_{1,2}$, $C_{3,4}$, and $C_{5,6}$ are not concurrent in $P$.
\end{proposition}

\begin{proof}Notice that none of the points $P_1,\ldots,P_8$ equals $P$: otherwise, there exists a subset of three of the $P_i$ such that they are aligned. Moreover, $P$ is not collinear with any two of the three points $P_1,P_3,P_5$, since this together with $P\in L_{1,2}\cap L_{3,4}\cap L_{5,6}$ would also contradict the general position of $P_1,\ldots,P_8$. Thus $P_1,P_3,P_5$ and $P$ are in general position, hence, after applying an automorphism of $\mathbb{P}^2$ if necessary, we may assume that we have $P = (0 : 0 : 1)$, and
$$P_1 =(0:1:1); \;\;\;P_3 =(1:0:1); \;\;\;P_5 =(1:1:1).$$
It follows that $L_{1,2}$ is the line given by $x = 0$, $L_{3,4}$ is the line given by $y = 0$, and $L_{5,6}$ is the line given by $x = y$.
Since $L_{7,8}$ is different from $L_{1,2},\; L_{3,4},\; L_{5,6}$ and contains $P$, there exists an $m\in k\setminus\{0,1\}$ such that $L_{7,8}$ is the line given by $my=x$. Therefore there are $a,b,c,d,e\in k\setminus\{0,1\}$ such that we can write

\begin{align}\label{eq:points}
& P_1 =(0:1:1); & &P_2 =(0:1:a);\\
&P_3=(1: 0: 1); & &P_4 =(1:0:b);\nonumber\\
&P_5=( 1: 1: 1 ); & &P_6 =(1:1:c);\nonumber\\
&P_7 =(m:1:d); & &P_8 =(m:1:e);\nonumber\\
&P\;=(0 : 0 : 1).\nonumber&&
\end{align}
We assume by contradiction that $C_{1,2}$, $C_{3,4}$, and $C_{5,6}$ contain $P$. Let Mon$_3$ be the decreasing sequence of monomials of
degree $3$ in $x, y, z$, ordered lexicographically with $x > y > z$, and for $j\in\{1\ldots 10\}$ let Mon$_i[j]$ be the $j^{\mbox{\tiny{th}}}$ entry of Mon$_i$. Let Mon$^1_i$ and Mon$^3_i$
be the list of derivatives of the entries in Mon$_i$ with respect to $S$ and $z$, respectively. For sequence of points in $\mathbb{P}^2$ given by $R=(R_1,\ldots,R_8)$, let $M_R$ be the matrix $$M_R=\left(c_{i,j}\right)_{i,j\in\{1,\ldots,10\}} \mbox{with } c_{i,j} = 
\left\{
	\begin{array}{ll}
		\mbox{Mon}_3[j](R_i)  & \mbox{for } i\leq 8 \\
		\mbox{Mon}_3^{x}[j](R_8)& \mbox{for } i = 9 \\
		\mbox{Mon}_3^{z}[j](R_8)& \mbox{for } i = 10 
	\end{array}
\right..$$ 
Then the point $P$ is on the cubic $C_{1,2}$ if and only if the determinant of $M_R$ is 0, where $R=(P_3,\ldots,P_8,P_1,P)$ \cite[Lemma 3.4 (iii)]{vLW}. We compute this determinant with \texttt{magma} (see \cite{MagmaCode}) and find $$\mbox{det}(M_R)=-m(m - 1)(d - e)(c - 1)(b - 1)(g_1b+g_2),$$ 
\noindent where $ g_1=cm^2 + de - dm - d - em - e - m^2 + 2m + 1$ and $g_2=-cde + cd + ce - c + 2de - 2d - 2e + 2$. 
    Note that  for each factor of det$(M_R)$ except the last one its vanishing implies  that $P_1,\ldots,P_8$ are not in general position, giving a contradiction. Therefore we find $g_1b+g_2=0$. Now assume that $g_1=0$, then it follows that $g_2=0$. But then we have $0=(c-1)g_1+g_2=(cm - e-m + 1)(cm -d- m + 1),$ and the latter is a product of two equations that each say that three of the points are collinear ($P_1,P_6,P_8$ and $P_1,P_6,P_7$, respectively). It follows that we have $g_1\neq0$ and we can write     \begin{equation}\label{conditiononb}b=-\frac{g_2}{g_1}=-\frac{(-cde + cd + ce - c + 2de - 2d - 2e + 2)}{(cm^2 + de - dm - d - em - e - m^2 + 2m + 1)}.
    \end{equation}
We now repeat this process twice and summarize what we find; for the detailed computations see \cite{MagmaCode}. The cubic $C_{3,4}$ passes through $P$ if and only if det$(M_R)=0$ with $R=(P_1,P_2,P_5\ldots,P_8,P_3,P)$. We factorize the determinant, and find $h_1a+h_2=0$ with $h_1=c + de - dm - d - em - e + m^2 + 2m - 1$ and $h_2= -cde + cdm + cem - cm^2 + 2de - 2dm - 2em + 2m^2$. From  $(c-1)h_1+h_2=(c-d+m-1)(c-e+m-1)$ and the fact that the latter factors imply that $P_3,P_6,P_7$ and $P_3,P_6,P_8$ are collinear, respectively, we conclude that $h_1\neq0$ and we obtain:
    \begin{equation}\label{conditionona}a=-\frac{h_2}{h_1}=-\frac{(-cde + cdm + cem - cm^2 + 2de - 2dm - 2em + 2m^2)}{(c + de - dm - d - em - e + m^2 + 2m - 1)}.
    \end{equation}

Setting the configuration with $a$ and $b$ as in (\ref{conditiononb}) and (\ref{conditionona}), we compute det$(M_R)$ with $R=(P_1,\ldots,P_4,P_7,P_8,P,P_5)$, which vanishes if and only if $P$ is contained in $C_{5,6}$. We obtain $h_1=c + de - dm - d - em - e + m^2 + 2m - 1=0$. But we already showed that this implies that the points are not in general position, giving a contradiction. We conclude that the cubics $C_{1,2},C_{3,4},C_{5,6}$ do not all go through~$P$.
\end{proof}

We are now ready to prove Theorem \ref{thm: case8lines}.

\begin{proofofthm8lines}
Let $e_1,\ldots,e_8$ be 8 exceptional curves that intersect in a point $P$ on~$S$. If there are $i,j\in\{1,\ldots,8\}$ such that $e_i\cdot e_j=3$, then $P$ lies on the ramification curve of the morphism $\varphi$ by \cite[Remark 2.11]{vLW}, hence $P_{\E}$ is torsion on its fiber by Remark \ref{remramcurve}. Assume that all $e_1,\ldots,e_8$ pairwise intersect with multiplicities 1 and 2. From Proposition \ref{prop: 47 cliques of size 8} it follows that the isomorphism type of their weighted intersection graph is one of the 45  different graphs in Figure~\ref{table cliques}. Moreover, from Proposition \ref{prop: n-gon torsion} we conclude that cliques in $G$ with isomorphism types 24-47 in Figure \ref{table cliques} correspond to exceptional curves that, if concurrent, intersect in a point which is torsion on its fiber. Finally, for cliques in $G$ with isomorphism types 9, 16, 17, 18, 20, 21, 22, and 23 in Figure \ref{table cliques}, we check that the Gram matrix of the corresponding sections in the Mordell-Weil group of $\E$ has at least one vector in the kernel whose entries do not sum to 0 \cite{MagmaCode}. Completely analogous to the proof of Theorem \ref{thmtorsion}, it follows that if the exceptional curves corresponding to one of these cliques were concurrent in a point $P$, the point $P_{\E}$ would be torsion on its fiber, since the corresponding sections specialize to a relation $aP_{\E}=0$ for some $a\neq0$. This proves the first part of the theorem. 

For the second part, let $Y$ be a del Pezzo surface of degree 1 over an field $k$ of characteristic 0. Without loss of generality we can assume that $k$ is algebraically closed, and that $Y$ is the blow-up of points $P_1,\ldots,P_8\in\mathbb{P}^2$ in general position. Let $G'$ be the weighted intersection graph of the 240 exceptional curves on $Y$. Let $K_1=\{f_1,\ldots,f_8\}$ be the clique in $G'$ corresponding to the exceptional curves on $Y$ given by the strict transforms of $L_{1,2},L_{3,4},L_{5,6},L_{7,8},C_{1,2},C_{3,4},C_{5,6},C_{7,8}\subset\mathbb{P}^2,$
and $K_2$ the clique corresponding to the strict transforms of 
$L_{1,2},L_{3,4},L_{5,6},L_{7,8},C_{1,2},C_{3,4},C_{5,6},Q_{2,4,7}\subset\mathbb{P}^2.$ All exceptional curves in $K_1$ intersect pairwise with multiplicity 1, so the intersection graph of $K_1$ is equal to number 7 in Figure~\ref{table cliques}. Similarly, all exceptional curves in $K_2$ intersect pairwise with multiplicity 1, except the pairs $\{L_{5,6},Q_{2,4,7}\},\;\{C_{1,2},Q_{2,4,7}\},\;\{C_{3,4},Q_{2,4,7}\}$, which all intersect with multiplicity 2. Therefore, the intersection graph of $K_2$ equals number 15 in Figure \ref{table cliques}. 

Let $e_1,\ldots,e_8$ be exceptional curves on $Y$ that pairwise intersect with multiplicities 1 and 2, and let $C$ be their weighted intersection graph. Assume that $C$ equals number 7 in Figure \ref{table cliques}. By Proposition~\ref{prop: 47 cliques of size 8}, there is only one $W_8$-orbit of cliques of 
size 8 in $G$ with intersection graph equal to number 7. Therefore, after 
permuting the indices if necessary, there is an element $w\in W_8$ such that 
$e_i=w(f_i)$ for $i\in\{1,\ldots,8\}$. Write 
$E_i'=w(E_i)$ for $i$ in $\{1,\ldots,8\}$, where $E_i$ is as in Notation~\ref{notation curves P^2}. Then, since the $E'_i$ are pairwise disjoint, $Y$ is isomorphic to the blow-up of $\mathbb{P}^2$ in points $Q_1,
\ldots,Q_8$ in $\mathbb{P}^2$ such that $E_i'$ is the exceptional curve above $Q_i$ for all $i$ \cite[Lemma 2.4]{vLW}. It follows that, under this blow-up, 
the exceptional curves $e_1,\ldots,e_8$ are the strict transforms of the curves $L_{1,2},L_{3,4},L_{5,6},L_{7,8},C_{1,2},C_{3,4},C_{5,6},C_{7,8}$ in $\mathbb{P}^2$, but now defined with respect to the points $Q_1,\ldots,Q_8$. Assume that the lines $L_{1,2},L_{3,4},L_{5,6},L_{7,8}$ with respect to $Q_1,\ldots,Q_8$ are concurrent in a point $P$. Then by 
Proposition \ref{propno3cubics}, the curves $C_{1,2},C_{3,4},C_{5,6}$ with respect to $Q_1,\ldots,Q_8$ do not go through $P$. We conclude that the exceptional curves in a clique with isomorphism type 7 are not concurrent. The same holds completely analogously for clique number 15. This finishes the proof. \qed
\end{proofofthm8lines}

\bibliographystyle{amsalpha}
\bibliography{bibliography}

\end{document}